\documentclass[12pt,final]{amsart}
\usepackage[utf8]{inputenc}
\usepackage[T1]{fontenc}
\usepackage{lmodern}
\usepackage{amsmath,amssymb,amsthm,mathtools}
\usepackage[margin=1in]{geometry}
\usepackage{mathrsfs,microtype}
\usepackage[shortlabels]{enumitem}
\usepackage[ocgcolorlinks,linkcolor=blue,citecolor=blue,urlcolor=blue]{hyperref}
\hypersetup{pdftitle={Global recovery of a scalar diffusivity from a single final state},pdfauthor={Ching-Lung Lin and Yi-Hsuan Lin}}
\numberwithin{equation}{section}
\allowdisplaybreaks[2]
\mathtoolsset{showonlyrefs}
\newtheorem{theorem}{Theorem}[section]
\newtheorem{lemma}[theorem]{Lemma}
\newtheorem{proposition}[theorem]{Proposition}
\newtheorem{corollary}[theorem]{Corollary}
\theoremstyle{definition}

\newtheorem{remark}[theorem]{Remark}
\newcommand{\R}{\mathbb R}
\newcommand{\C}{\mathbb C}
\newcommand{\Om}{\Omega}

\newcommand{\norm}[1]{\lVert #1\rVert}
\newcommand{\ip}[2]{\langle #1,#2\rangle}
\newcommand{\e}{\mathrm e}
\newcommand{\ii}{\mathsf{i}}
\newcommand{\para}[1]{\vspace{3mm}\noindent\textbf{#1.}}
\DeclareMathOperator{\Div}{div}
\DeclareMathOperator{\osc}{osc}
\DeclareMathOperator{\supp}{supp}

\title[Diffusivity from a single final state]{Global recovery of a scalar diffusivity from a single final state}

\author[C.-L. Lin]{Ching-Lung Lin}
\address{Department of Mathematics, National Cheng Kung University, Tainan 701, Taiwan.}
\email{cllin2@mail.ncku.edu.tw}

\author[Y.-H. Lin]{Yi-Hsuan Lin}
\address{Department of Applied Mathematics, National Yang Ming Chiao Tung University, Hsinchu, Taiwan \& Fakult\"at f\"ur Mathematik, University of Duisburg-Essen, Essen, Germany}
\email{yihsuanlin3@gmail.com, yihsuanlin3@nycu.edu.tw}

\subjclass[2020]{Primary 35R30; Secondary 35K05, 35K20, 35B35}
\keywords{Inverse parabolic problem, scalar diffusivity, final state, global uniqueness, Lipschitz stability, experimental design}
\date{}

\begin{document}
	
	\begin{abstract}
		We prove global uniqueness and Lipschitz stability for a scalar diffusivity in the heat equation from a single final state at any prescribed positive time. On each prescribed uniformly elliptic $C^{1,\alpha}$-bounded class with known coefficient boundary values, we choose a constant initial datum and a spatially constant exponential boundary input that work for the entire class. The result holds on smooth bounded connected domains in every dimension $n\ge2$ and gives an affirmative answer to \cite[Problem~8.2]{Alessandrini} within these classes. The stability estimate controls the $L^2$ diffusivity difference by the $H^2$ final-state difference, and this data norm is optimal within the Sobolev scale for the chosen experiment.
	\end{abstract}
	
	\maketitle
	
	\section{Introduction}\label{sec:introduction}
	
	Let $\Om\subset\R^n$, $n\ge2$, be a bounded connected domain with smooth boundary, and fix $T>0$. We consider the recovery of a positive real scalar diffusivity $a=a(x)$ in
	\begin{equation}\label{eq:general-heat}
		\begin{aligned}
			\begin{cases}
				\partial_tu-\Div(a\nabla u)=0 &\text{in }\Om\times(0,\infty),\\
				u=g &\text{on }\partial\Om\times(0,\infty),\\
				u(\cdot,0)=h &\text{in }\Om.
			\end{cases}
		\end{aligned}
	\end{equation}
	The initial datum $h$ and boundary input $g$ are prescribed, the boundary values of $a$ are known, and the observation is the single final state $u(\cdot,T)$. We write $u(t)=u(\cdot,t)$, and use the same convention for the functions below. The question is whether one can choose $h$ and $g$ independently of the unknown coefficient so that this observation determines $a$. This is the final-data problem in \cite[Section~8, Problem~8.2]{Alessandrini}.

	The main difficulty is that the equation at the observation time contains the unmeasured derivative $\partial_tu(T)$. Thus, the final state alone does not provide the right-hand side of an elliptic coefficient equation. A coefficient difference also appears inside a divergence, while the desired estimate controls only its $L^2$ norm. Critical points of the final state can obstruct a direct transport argument. Moreover, comparing two unknown coefficients changes the evolution operator itself, so the inverse problem remains nonlinear despite the linearity of the forward equation in $u$.

	\subsection{Related results} 
	
	General accounts of inverse parabolic problems, including coefficient and source recovery, are given in \cite[Chapter~3]{Choulli2009} and \cite[Chapter~9]{Isakov2017}. Final overdetermination for parabolic equations has been studied in \cite{AlessandriniVessella,Isakov}. The formulation in \cite[Section~8]{Alessandrini} distinguishes two questions. Problem~8.1 concerns recovery of a scalar diffusivity in two dimensions from a known source, zero initial and boundary data, and observations at one or more positive times. Problem~8.2 instead asks whether an appropriate choice of the initial datum and boundary input permits recovery from one final state when the coefficient boundary values are known. The latter formulation chooses the experiment part of the uniqueness question.

	The elliptic counterpart is the recovery of a diffusion coefficient from a known solution inside the domain. A first-order equation for the coefficient is used in \cite{Richter1981} to obtain continuous dependence under suitable assumptions. In two dimensions, \cite{Alessandrini1986} studies this identification problem with particular attention to critical points of the observed solution. For elliptic Dirichlet equations with a strictly positive source, \cite{BonitoCohenDeVorePetrovaWelper2017} obtains weighted energy estimates and conditional H\"older stability for the $L^2$ coefficient difference from $H^1$ interior data under a priori coefficient bounds. In the parabolic problem, applying an elliptic coefficient estimate at time $T$ also requires control of the unobserved term $\partial_tu(T)$ and its dependence on the diffusivity.
	
	With the principal operator prescribed, \cite[Theorem~4.1(i)--(ii), equation~(4.4)]{Isakov} establishes global uniqueness and Lipschitz stability in Hölder norms for a nonnegative zeroth-order coefficient under positivity and temporal monotonicity assumptions on the boundary input. The stability estimate also requires the final states to be bounded below by a positive constant. For the diffusivity, \cite[Corollary~4.2, equation~(4.11)]{Isakov} proves uniqueness on an interval, with zero initial data, a prescribed coefficient value at one endpoint, homogeneous Neumann data at that endpoint, and a nonzero Dirichlet input at the other endpoint that is nondecreasing and convex in time and satisfies the stated compatibility conditions.
	
	An early study of zeroth-order coefficient recovery from initial and final states is \cite{Rundell1987}. For a semilinear equation with a prescribed nonlinearity, \cite[Theorem~1]{Choulli1994} proves existence for the recovery of a spatial zeroth-order coefficient on convex domains under small-time and data assumptions. Its continuous dependence estimate in \cite[Theorem~2]{Choulli1994} does not require convexity. With nonzero initial data, \cite{ChoulliYamamoto1997} proves local well-posedness near the zero coefficient under an a priori restriction on its support and additional assumptions on the data. In these problems, the principal diffusion operator is prescribed.
	
	For homogeneous Dirichlet data and zero source, \cite[Definition~1.1 and Theorem~1.1]{Triki} proves global $L^2/H^2$ Lipschitz stability in dimensions two and three for sufficiently large observation times. The coefficients have prescribed boundary values and a uniform $C^1$ bound, together with a common quantitative gap condition between distinct Dirichlet eigenvalues. The initial datum satisfies a quantitative dominance condition on its weighted positive and negative parts; in particular, every nonzero nonnegative $L^2$ initial datum is admissible; see \cite[equation~(3)]{Triki}. The proof separates the first spectral mode from a faster decaying remainder and estimates the change in the spectral projections as the diffusivity varies. The decay of the remainder permits absorption for large time. 
	
	For equations with a known time-independent source, \cite[Assumption~1 and Theorem~1]{JinLuQuanZhou} gives a weighted estimate for the diffusivity difference on smooth convex domains in dimensions at most three. The estimate contains a term controlled by the final-state difference and a remainder that decays with the observation time. Under the positivity condition in \cite[equation~(2.7)]{JinLuQuanZhou}, sufficiently large time gives conditional H\"older stability, with the $L^2$ coefficient difference bounded by the square root of the $H^1$ seminorm of the final-state difference. The argument applies to both standard and time-fractional diffusion and is also used to analyze a numerical reconstruction. 
	
	Final and integral observations are also used in \cite{PrilepkoKostin1993} to study inverse source problems with a prescribed principal operator. For sources of a specified form, that work establishes the Fredholm property and gives existence, uniqueness, and stability results. A semigroup approach to source recovery from final data is developed in \cite{HasanovSlodicka2013}. Its representation formula describes the possible nonuniqueness and provides sufficient conditions for uniqueness.
	
	With the diffusion parameter and the remaining forward coefficients prescribed, \cite[main theorem, p.~243]{ChoulliYamamoto} determines a spatial source factor from one final state. The known factor multiplying the source is positive at the final time, and the stability estimate is in Hölder norms. For every fixed $s_0>0$, the result holds for all diffusion parameters $s>s_0$ outside a finite exceptional set. The proof writes the source equation as a compact perturbation of an invertible multiplication operator. The semigroup estimate in \cite[Lemma~5]{ChoulliYamamoto} gives a contraction for large $s$, and analytic Fredholm theory extends invertibility to the stated parameter set.
	
	A single final state is used in \cite{ChoulliYamamoto2008} to determine a zeroth-order coefficient, an initial state, and a Robin boundary coefficient. The zeroth-order coefficient is assumed known in an interior subdomain. The initial state is nonnegative, nonzero, and compactly supported, and the known source satisfies sign and time analyticity conditions. The principal diffusion operator is fixed.
	
	Existence for a one-dimensional terminal diffusivity identification problem is studied in \cite{Kozhanov}. Simultaneous determination of a time-independent source and a diffusivity is considered in \cite{DeCezaroJohansson}. Its multidimensional uniqueness result assumes that the diffusivity is known in a neighborhood of the boundary, together with regularity bounds on the solution and differential conditions on the final profile; see \cite[Assumption~1 and Theorem~5]{DeCezaroJohansson}. For a prescribed principal part, \cite[Theorem~1.4]{ImanuvilovYamamoto} proves global $L^2$ Lipschitz stability for a zeroth-order coefficient. The estimate uses $H^4$ norms of the initial and final state differences and an $L^2$ norm of the time derivative of the conormal flux difference on a relatively open boundary portion near both time endpoints. The coefficient boundary values are prescribed, the compared solutions have the same Dirichlet boundary values, and one of them is assumed to stay uniformly away from zero near both time endpoints.
	
	Carleman estimates provide another classical approach to inverse parabolic problems; see \cite[Sections~6 and~8.2]{Yamamoto2009}. For the recovery of an anisotropic principal coefficient, \cite[Theorems~1.1--1.2]{YuanYamamoto2009} chooses finitely many interior sources using a prescribed reference coefficient. Approximate controllability produces intermediate states satisfying a nondegeneracy condition, and a Carleman estimate gives Lipschitz stability from these spatial states together with time-dependent normal derivatives on a boundary subset. The coefficients satisfy a common regularity bound and are prescribed in a neighborhood of the boundary.

	\subsection{Main result}
	
	We prove global uniqueness and Lipschitz stability from one final state in every dimension $n\ge2$, giving an affirmative answer to \cite[Problem~8.2]{Alessandrini} on each prescribed uniformly elliptic $C^{1,\alpha}$-bounded class with known coefficient boundary values. We take $h=1$ and $g=\e^{\lambda t}$, with $\lambda$ chosen from the domain, the coefficient bounds, and the prescribed positive observation time.
	
	We write $\norm{\cdot}_p$ for the $L^p(\Om)$ norm and take all spatial norms without an indicated domain on $\Om$. Fix $0<\alpha<1$, $0<\kappa<1$, and $0<M<\infty$, and set
	\begin{equation}\label{eq:class}
		\begin{aligned}
			\mathcal A_{\kappa,M}=\{a\in C^{1,\alpha}(\overline\Om;\R):\kappa\le a\le\kappa^{-1},\ \norm{a}_{C^{1,\alpha}(\overline\Om)}\le M\}.
		\end{aligned}
	\end{equation}
	We use a norm that dominates $\norm{a}_\infty$ and $\norm{\nabla a}_\infty$. The bounds $\kappa$ and $M$ are known before the experiment. For $\lambda>0$, consider
	\begin{equation}\label{eq:experiment}
		\begin{aligned}
			\begin{cases}
				\partial_tu-\Div(a\nabla u)=0 &\text{in }\Om\times(0,\infty),\\
				u=\e^{\lambda t} &\text{on }\partial\Om\times(0,\infty),\\
				u(\cdot,0)=1 &\text{in }\Om.
			\end{cases}
		\end{aligned}
	\end{equation}
	The initial and boundary values agree at $t=0$. The solution is constructed in Section~\ref{sec:normalized}, where its spatial regularity at positive times is also established.
	
	\begin{theorem}\label{thm:global}
		For every $T_*>0$, there are $\lambda,C_*>0$, depending only on $\Om,n,\alpha,\kappa,M,T_*$, with the following property. If $a_1,a_2\in\mathcal A_{\kappa,M}$ have the same boundary trace and $u_j$ solve \eqref{eq:experiment} with this $\lambda$, then
		\begin{equation}\label{eq:global-stability}
			\begin{aligned}
				\norm{a_1-a_2}_{L^2(\Om)}\le C_*\norm{u_1(T)-u_2(T)}_{H^2(\Om)},\quad T\ge T_*.
			\end{aligned}
		\end{equation}
		In particular, equality of the two final states at any one time $T\ge T_*$ implies $a_1=a_2$. The same exponential rate and stability constant work for every admissible coefficient pair and every $T\ge T_*$.
	\end{theorem}
	
	For any prescribed observation time $T>0$, take $T_*=T$. The input is independent of the unknown coefficient within each prescribed class, but may depend on its common regularity bound $M$.
	
	The choice of a constant initial datum is not essential. Corollary~\ref{cor:positive-inputs} allows any prescribed positive spatial profile, with a common exponential rate for profiles having fixed lower and regularity bounds. We also obtain a different type of experiment near a known smooth positive reference coefficient: Theorem~\ref{local:thm:main} gives pairwise Lipschitz stability with a time-independent boundary input. Both unknown coefficients may vary near the reference, and the same input is used for the pair.
	
	Theorems~\ref{thm:exponential-sharpness} and~\ref{affine:thm:sharpness} show that, for the fixed exponential and affine experiments, no uniform $L^2$ Lipschitz estimate with $H^s$ final data, $0\le s<2$, holds even in an arbitrarily small $C^{1,\alpha}$ neighborhood of the constant coefficient $1$.

	\subsection{Ideas of the proof}
	
	The exponential input introduces a parameter that we can choose before the experiment. Multiplying the solution by $\e^{-\lambda t}$ makes its boundary value constant and shifts the Dirichlet diffusion operator by $\lambda$. Subtracting the normalized solution from $1$ then gives a profile with zero initial and boundary values and a known constant forcing term. Its evolution formula expresses the unobserved time derivative at $T$ as a function of the shifted operator applied to the final profile. This yields an exact equation for the final state, as shown in Section~\ref{sec:normalized}.
	
	An elliptic barrier gives a positive lower bound for the Dirichlet operator applied to this profile. When two coefficients have the same boundary values, testing their difference equation with the coefficient difference divided by one of the coefficients removes the boundary terms. Integration by parts then turns the positive lower bound into an $L^2$ estimate for the coefficient difference. This argument allows the difference to change sign and the final profile to have critical points.
	
	The remaining issue is to compare the operator functions associated with the two unknown diffusivities. Section~\ref{sec:multiplier} uses the resolvent identity on a contour common to the coefficient class. The coefficient difference is kept inside a divergence, so the estimate involves its $L^2$ norm without differentiating it. The resulting bound decays exponentially with the input rate at every fixed positive time. Since the argument estimates the operator function directly, it does not require a separation of individual eigenvalues.
	
	The decay in this operator estimate dominates the loss in the elliptic lower bound as $\lambda$ increases. A rate chosen uniformly over the coefficient class therefore allows the term involving the unknown coefficient difference to be absorbed. The remaining term is controlled by the $H^2$ difference of the measured final states. Section~\ref{sec:global-proof} carries out this comparison and selects the input rate.

	\para{Organization of the paper} In Section~\ref{sec:global-recovery}, we prove Theorem~\ref{thm:global}. Section~\ref{sec:positive-inputs} extends the experiment to positive spatial profiles. Section~\ref{sec:local} proves local pairwise stability with time-independent inputs near a known reference. Finally, we prove optimality of the $H^2$ final-data norm for the exponential and affine experiments in Section~\ref{sec:optimality}.

	\section{Global recovery}\label{sec:global-recovery}
	
	\subsection{Uniform elliptic and positive-time estimates}\label{sec:uniform}
	
	All constants in this subsection are uniform for $a\in\mathcal A_{\kappa,M}$. The inner product on complex $L^2(\Omega)$ is linear in its first argument.
	
	The closed form $\mathfrak q_a(v,w)=\int_\Om a\nabla v\cdot\nabla\overline w\,dx$, with domain $H_0^1(\Om)$, defines the positive self-adjoint Dirichlet operator $L_a=-\Div(a\nabla)$. Since $H_0^1(\Om)$ embeds compactly into $L^2(\Om)$, this operator has compact resolvent. Uniform ellipticity and Poincar\'e's inequality give a constant $\rho_*>0$, depending only on $\Om$ and $\kappa$, such that
	\begin{equation}\label{eq:spectral-lower}
		\begin{aligned}
			\inf\sigma(L_a)\ge\rho_*
			\quad\text{for every }a\in\mathcal A_{\kappa,M}.
		\end{aligned}
	\end{equation}
	For a function with nonzero boundary values, we write $-\Div(a\nabla v)$ directly.
	
	We use the Dirichlet regularity results in \cite[Theorems~6.6, 8.12, 9.13 and~9.15]{GilbargTrudinger2001}. The weak $H^2$ regularity theorem applies to $-\Div(a\nabla v)$, while the $W^{2,p}$ and Schauder estimates apply to $-a\Delta v-\nabla a\cdot\nabla v$. Their constants are controlled by the prescribed coefficient bounds and the fixed domain. We record the consequences needed below.
	
	\begin{lemma}\label{lem:elliptic}
		For every $a\in\mathcal A_{\kappa,M}$,
		\begin{equation}\label{eq:domain}
			\begin{aligned}
				\operatorname{Dom}(L_a)=H^2(\Om)\cap H_0^1(\Om).
			\end{aligned}
		\end{equation}
		On this domain,
		\[
		\begin{aligned}
			C^{-1}\norm{v}_{H^2}\le\norm{v}_2+\norm{L_av}_2\le C\norm{v}_{H^2}.
		\end{aligned}
		\]
		For functions with zero boundary trace and the regularity indicated on the left,
		\begin{equation}\label{eq:w2p}
			\begin{aligned}
				\norm{v}_{W^{2,p}}\le C_p(\norm{L_av}_p+\norm{v}_p),\quad 2\le p<\infty,
			\end{aligned}
		\end{equation}
		and
		\begin{equation}\label{eq:schauder}
			\begin{aligned}
				\norm{v}_{C^{2,\alpha}}\le C_\alpha(\norm{L_av}_{C^\alpha}+\norm{v}_\infty).
			\end{aligned}
		\end{equation}
		If $v\in H_0^1(\Om)$ solves $-\Div(a\nabla v)=f$ weakly, then $v\in W^{2,p}$ for $f\in L^p$, $2\le p<\infty$, and $v\in C^{2,\alpha}$ for $f\in C^\alpha$, with the corresponding estimates above. The constants depend only on $\Om,n,\kappa,M$ and, where applicable, $p,\alpha$.
	\end{lemma}
	
	The semigroup preserves positivity and satisfies
	\[
	\begin{aligned}
		\norm{\e^{-tL_a}f}_\infty\le\norm{f}_\infty,\quad t\ge0;
	\end{aligned}
	\]
	see \cite[Corollary~4.3 and Theorem~4.6]{Ouhabaz2005}. We use the same notation for its consistent extension to $L^p(\Om)$. By \cite[Theorem~3.13]{Ouhabaz2005}, this extension is holomorphic and contractive on a sector whose angle depends only on $p$, for $1<p<\infty$. Cauchy's formula, as in \cite[equation~(1.19)]{Ouhabaz2005}, gives
	\[
	\begin{aligned}
		\norm{\e^{-tL_a}f}_p+t\norm{L_a\e^{-tL_a}f}_p\le C_p\norm{f}_p,\quad t>0.
	\end{aligned}
	\]
	Here $L_a$ on $L^p$ denotes the negative of the generator of the consistent extension. Since the sector angle and the contraction bound are independent of $a$, so is $C_p$. Taking $p=2n$ and applying \eqref{eq:w2p} and Sobolev embedding gives the uniform $L^\infty$ to $C^\alpha$ estimate at times $t\ge T_*>0$.
	
	For bounded $f$, elliptic regularity gives $v=L_a^{-1}f\in W^{2,2n}(\Om)\cap H_0^1(\Om)$, satisfying $a\Delta v+\nabla a\cdot\nabla v=-f$. The weak maximum principle of \cite[Theorem~9.1, equation~(9.4)]{GilbargTrudinger2001} gives $\norm{v}_\infty\le E\norm{f}_\infty$, with $E$ depending only on $\Om,n,\kappa,M$. For complex-valued $f$, apply it to the real and imaginary parts and increase $E$ by a fixed factor if necessary.
	
	\begin{lemma}\label{lem:positive-time}
		For every $T_*>0$, there are constants $E,H_*,B_0<\infty$, depending only on $\Om,n,\alpha,\kappa,M,T_*$, such that
		\begin{equation}\label{eq:inverse-sup}
			\begin{aligned}
				\norm{L_a^{-1}f}_\infty\le E\norm{f}_\infty,
			\end{aligned}
		\end{equation}
		\begin{equation}\label{eq:heat-smoothing}
			\begin{aligned}
				\norm{\e^{-TL_a}f}_{C^\alpha}\le H_*\norm{f}_\infty,\quad T\ge T_*,
			\end{aligned}
		\end{equation}
		and
		\begin{equation}\label{eq:correction-smoothing}
			\begin{aligned}
				\norm{L_a^{-1}(I-\e^{-TL_a})f}_{C^{2,\alpha}}\le B_0\norm{f}_{C^\alpha},\quad T\ge T_*.
			\end{aligned}
		\end{equation}
		The first two estimates hold for $f\in L^\infty(\Om)$, and the third for $f\in C^\alpha(\overline\Om)$. The constant $E$ is independent of $T_*$. No boundary condition on $f$ is required.
	\end{lemma}
	
	To obtain \eqref{eq:correction-smoothing}, set $z=L_a^{-1}(I-\e^{-TL_a})f$, then $L_az=f-\e^{-TL_a}f$ and $\norm{z}_\infty\le2E\norm{f}_\infty$. Applying \eqref{eq:schauder} and \eqref{eq:heat-smoothing} gives
	\[
	\begin{aligned}
		\norm{z}_{C^{2,\alpha}}\le C_\alpha(1+H_*+2E)\norm{f}_{C^\alpha}.
	\end{aligned}
	\]
	Thus, one may take $B_0=C_\alpha(1+H_*+2E)$.

	\subsection{The final-state equation and positivity}\label{sec:normalized}
	
	Fix $\lambda>0$ and $a\in\mathcal A_{\kappa,M}$, and set $A_a=L_a+\lambda I$. For the experiment \eqref{eq:experiment}, write
	\begin{equation}\label{eq:normalization-intro}
		\begin{aligned}
			W_a(t)=\e^{-\lambda t}u_a(t),\quad P_a(t)=1-W_a(t).
		\end{aligned}
	\end{equation}
	Substitution of $u_a=\e^{\lambda t}(1-P_a)$ gives
	\begin{equation}\label{eq:normalized-pde}
		\begin{aligned}
			\begin{cases}
				\partial_tP_a+A_aP_a=\lambda &\text{in }\Om\times(0,\infty),\\
				P_a=0 &\text{on }\partial\Om\times(0,\infty),\\
				P_a(\cdot,0)=0 &\text{in }\Om.
			\end{cases}
		\end{aligned}
	\end{equation}
	
	For $r>0$ and $f\in L^2(\Om)$, the function $\e^{-rA_a}f$ is the solution at time $r$ of $\partial_tv+A_av=0$ with homogeneous Dirichlet boundary conditions and initial datum $f$. The boundary condition holds at positive times, while the initial datum is attained in $L^2(\Om)$; thus, $f$ need not have zero boundary trace. Duhamel's formula gives
	\begin{equation}\label{eq:mild}
		\begin{aligned}
			P_a(t)=\lambda\int_0^t\e^{-rA_a}1\,dr=\lambda A_a^{-1}(I-\e^{-tA_a})1.
		\end{aligned}
	\end{equation}
	Formula \eqref{eq:mild} defines $P_a(t)\in\operatorname{Dom}(L_a)$ for $t>0$, with $P_a(t)\to0$ in $L^2(\Om)$ as $t\to 0$. Differentiation gives $\partial_tP_a(t)=\lambda\e^{-tA_a}1$, for $t>0$. Since $A_aP_a(t)=\lambda(I-\e^{-tA_a})1$, these identities show that $P_a$ solves \eqref{eq:normalized-pde}, and $u_a(t)=\e^{\lambda t}(1-P_a(t))$ solves \eqref{eq:experiment}. Positivity and $L^\infty$ contraction give $0\le\e^{-rL_a}1\le1$. Using $\e^{-rA_a}=\e^{-\lambda r}\e^{-rL_a}$ in \eqref{eq:mild}, we obtain $0\le P_a(t)\le1-\e^{-\lambda t}\le1$.
	
	Applying $L_a=A_a-\lambda I$ to \eqref{eq:mild} yields
	\begin{equation}\label{eq:lp-formula}
		\begin{aligned}
			L_aP_a(T)=\lambda(1-P_a(T)-\e^{-TA_a}1).
		\end{aligned}
	\end{equation}
	The right-hand side of \eqref{eq:lp-formula} is bounded. Hence, Lemma~\ref{lem:elliptic} gives $P_a(T)\in W^{2,p}(\Om)$ for every finite $p\ge2$. Taking $p=2n$ and using Sobolev embedding gives $P_a(T)\in C^\alpha(\overline\Om)$. Also, $\e^{-TA_a}1=\e^{-\lambda T}\e^{-TL_a}1\in C^\alpha(\overline\Om)$ by Lemma~\ref{lem:positive-time}. Applying \eqref{eq:schauder} to \eqref{eq:lp-formula} now gives $P_a(T)\in C^{2,\alpha}(\overline\Om)$. For a fixed $\lambda$ and $T_*>0$, the resulting bound for $P_a(T)$ is uniform in $a\in\mathcal A_{\kappa,M}$ and $T\ge T_*$. The final state $u_a(T)=\e^{\lambda T}(1-P_a(T))$ also belongs to $C^{2,\alpha}(\overline\Om)$.
	
	The function $v_a(t)=\e^{\lambda t}-u_a(t)$ has zero initial and boundary values and satisfies $\partial_tv_a-\Div(a\nabla v_a)=\lambda\e^{\lambda t}$. Thus, the prescribed experiment can also be written as a problem with a known source independent of $a$. The final-state differences satisfy $v_{a_1}(T)-v_{a_2}(T)=-(u_{a_1}(T)-u_{a_2}(T))$, and $P_a(t)=\e^{-\lambda t}v_a(t)$.
	
	\begin{lemma}\label{lem:final-identities}
		For $T>0$, define $F_T(s)=s/(\e^{Ts}-1)$ for $s>0$ and set $S_a^\lambda(T)=F_T(A_a)$. This operator is defined by expansion in an orthonormal eigenbasis of $A_a$: if $A_a\phi=s\phi$, then $S_a^\lambda(T)\phi=F_T(s)\phi$. One has
		\begin{equation}\label{eq:final-equation}
			\begin{aligned}
				(A_a+S_a^\lambda(T))P_a(T)=\lambda,
			\end{aligned}
		\end{equation}
		\begin{equation}\label{eq:multiplier-action}
			\begin{aligned}
				S_a^\lambda(T)P_a(T)=\lambda\e^{-TA_a}1=\partial_tP_a(T),
			\end{aligned}
		\end{equation}
		and
		\begin{equation}\label{eq:multiplier-norm}
			\begin{aligned}
				\norm{S_a^\lambda(T)}_{2\to2}\le T^{-1}.
			\end{aligned}
		\end{equation}
	\end{lemma}
	
	\begin{proof}
		For every $s>0$, $F_T(s)\frac{1-\e^{-Ts}}s=\e^{-Ts}$. Applying this scalar identity to \eqref{eq:mild} by spectral calculus gives
		\[
		\begin{aligned}
			S_a^\lambda(T)P_a(T)=\lambda F_T(A_a)A_a^{-1}(I-\e^{-TA_a})1=\lambda\e^{-TA_a}1=\partial_tP_a(T).
		\end{aligned}
		\]
		This proves \eqref{eq:multiplier-action}. Substituting it into \eqref{eq:normalized-pde} at $t=T$ proves \eqref{eq:final-equation}. Finally, $\e^{Ts}-1\ge Ts$ gives $0\le F_T(s)\le T^{-1}$, so the spectral theorem yields \eqref{eq:multiplier-norm}.
	\end{proof}
	
	\begin{lemma}\label{lem:positivity}
		Let $T_*>0$ and put
		\begin{equation}\label{eq:barrier-constants}
			\begin{aligned}
				d_0=\sqrt{2/\kappa},\quad D_0=d_0\osc_{\overline\Om}x_1.
			\end{aligned}
		\end{equation}
		If $\sqrt\lambda\ge d_0M$ and $\lambda T_*-D_0\sqrt\lambda\ge\log2$, then for every $a\in\mathcal A_{\kappa,M}$ and $T\ge T_*$,
		\begin{equation}\label{eq:positive-lp}
			\begin{aligned}
				L_aP_a(T)\ge\nu_\lambda:=\frac\lambda2\e^{-D_0\sqrt\lambda}\quad\text{on }\overline\Om.
			\end{aligned}
		\end{equation}
	\end{lemma}
	
	\begin{proof}
		Let $V_a\in C^{2,\alpha}(\overline\Om)$ solve
		\[
		\begin{aligned}
			\begin{cases}
				-\Div(a\nabla V_a)+\lambda V_a=0 &\text{in }\Om,\\
				V_a=1 &\text{on }\partial\Om.
			\end{cases}
		\end{aligned}
		\]
		Existence follows from \cite[Theorem~6.14]{GilbargTrudinger2001}, applied to $a\Delta+\nabla a\cdot\nabla-\lambda$. Comparison with the constant functions $0$ and $1$ gives $0\le V_a\le1$; see \cite[Theorem~3.3]{GilbargTrudinger2001}. Since $1-V_a$ has zero boundary values and satisfies $A_a(1-V_a)=\lambda$, one has $V_a=1-\lambda A_a^{-1}1$. Therefore, formulas \eqref{eq:normalization-intro} and \eqref{eq:mild} give
		\begin{equation}\label{eq:stationary-transient}
			\begin{aligned}
				W_a(T)=V_a+\e^{-TA_a}(1-V_a)\ge V_a,
			\end{aligned}
		\end{equation}
		where the inequality follows from the positivity of the semigroup.
		
		To obtain a lower bound for $V_a$, define the barrier function $\phi_\lambda(x)=\exp(d_0\sqrt\lambda(x_1-\max_{\overline\Om}x_1))$, then $\phi_\lambda\le1$ on $\partial\Om$ and
		\[
		\begin{aligned}
			-\Div(a\nabla\phi_\lambda)+\lambda\phi_\lambda=(\lambda-ad_0^2\lambda-d_0\sqrt\lambda\,\partial_1a)\phi_\lambda\le(-\lambda+d_0M\sqrt\lambda)\phi_\lambda\le0.
		\end{aligned}
		\]
		The comparison principle gives
		\begin{equation}\label{eq:barrier}
			\begin{aligned}
				V_a\ge\phi_\lambda\ge\e^{-D_0\sqrt\lambda}\quad\text{on }\overline\Om.
			\end{aligned}
		\end{equation}
		This construction uses only the range of $x_1$ on $\overline\Om$.
		
		By \eqref{eq:lp-formula}, \eqref{eq:stationary-transient}, \eqref{eq:barrier}, and $\e^{-TL_a}1\le1$, we obtain
		\[
		\begin{aligned}
			L_aP_a(T)=\lambda W_a(T)-\lambda\e^{-TA_a}1\ge\lambda(\e^{-D_0\sqrt\lambda}-\e^{-\lambda T}).
		\end{aligned}
		\]
		The condition $\lambda T_*-D_0\sqrt\lambda\ge\log2$ implies $\e^{-\lambda T}\le\tfrac12\e^{-D_0\sqrt\lambda}$ for $T\ge T_*$, proving \eqref{eq:positive-lp}.
	\end{proof}

	\subsection{Perturbation of the smoothing operator}\label{sec:multiplier}
	
	The next lemma estimates the difference of $S_a^\lambda(T)$ in terms of the $L^2$ difference of the coefficients. We include $\lambda=0$, which will be used for time-independent inputs in Section~\ref{sec:local}.
	
	\begin{lemma}\label{lem:multiplier}
		For every $T_*>0$, there is a constant $B_*>0$, depending only on $\Om,n,\alpha,\kappa,M,T_*$, such that
		\begin{equation}\label{eq:multiplier-bound}
			\begin{aligned}
				\norm{(S_{a_1}^\lambda(T)-S_{a_2}^\lambda(T))z}_2\le B_*(1+\lambda)\e^{-\lambda T}\norm{a_1-a_2}_2\norm{z}_\infty
			\end{aligned}
		\end{equation}
		for all $\lambda\ge0$, $T\ge T_*$, $a_1,a_2\in\mathcal A_{\kappa,M}$, and $z\in L^\infty(\Om;\C)$, where $S_a^\lambda(T)=F_T(L_a+\lambda I)$. The coefficients may have different boundary values.
	\end{lemma}
	
	\begin{proof}
		Fix $\theta=\pi/4$ and $r_0=\rho_*/2$, where $\rho_*$ is the lower bound in \eqref{eq:spectral-lower}, and write $R_a(\zeta)=(\zeta-L_a)^{-1}$. Since $\sigma(L_a)\subset[\rho_*,\infty)$, we estimate the distance $|\zeta-s|$ for real $s\ge\rho_*$. On the rays $\zeta=r\e^{\pm\ii\theta}$, $r\ge r_0$, one has
		\begin{equation}\label{eq:resolvent-geometry}
			\begin{aligned}
				|\zeta-s|\ge\sin(\theta/2)(r+s),\quad s\ge\rho_*.
			\end{aligned}
		\end{equation}
		Indeed,
		\[
		\begin{aligned}
			|r\e^{\pm\ii\theta}-s|^2=(r-s)^2+4rs\sin^2(\theta/2)\ge\sin^2(\theta/2)(r+s)^2.
		\end{aligned}
		\]
		Thus, the spectral theorem gives
		\begin{equation}\label{eq:resolvent-estimate}
			\begin{aligned}
				\norm{R_a(\zeta)}_{2\to2}
				=\sup_{s\in\sigma(L_a)}\frac1{|\zeta-s|}
				\le\frac1{\sin(\theta/2)(r+\rho_*)}
				\le\frac{C_\theta}{r},
			\end{aligned}
		\end{equation}
		where $C_\theta=1/\sin(\theta/2)$.
		
		Fix $a\in\mathcal A_{\kappa,M}$ and $z\in L^\infty(\Om)$, and let $\zeta=r\e^{\pm\ii\theta}$ with $r\ge r_0$. Set $v=R_a(\zeta)z$, then $v\in\operatorname{Dom}(L_a)$ and $L_av=\zeta v-z$. The bound \eqref{eq:resolvent-estimate} gives $\norm{v}_2\le(C_\theta/r)\norm{z}_2$. Applying \eqref{eq:w2p} with $p=2$, we obtain
		\[
		\begin{aligned}
			\norm{v}_{W^{2,2}}&\le C(\norm{L_av}_2+\norm{v}_2)\le C((1+r)\norm{v}_2+\norm{z}_2)\\
			&\le C\Bigl(1+C_\theta\frac{1+r}{r}\Bigr)\norm{z}_2\le C(1+C_\theta(1+r_0^{-1}))|\Om|^{1/2}\norm{z}_\infty,
		\end{aligned}
		\]
		where we used $(1+r)/r\le1+r_0^{-1}$ and $\norm{z}_2\le|\Om|^{1/2}\norm{z}_\infty$. Since $\theta$ and $r_0$ are fixed, absorbing these factors into $C$ gives $\norm{v}_{W^{2,2}}\le C\norm{z}_\infty$, with $C$ independent of $a$, $r$, and $z$.
		
		Set $m=n-1$ and $p_\ell=2n/(n-\ell)$ for $0\le\ell\le m$. Thus, $p_0=2$ and $p_m=2n$. For $0\le\ell<m$, the identity $1/p_\ell-1/p_{\ell+1}=1/(2n)$ gives the Sobolev embedding $W^{2,p_\ell}(\Om)\hookrightarrow W^{1,p_{\ell+1}}(\Om)$. Suppose that $\norm{v}_{W^{2,p_\ell}}\le C(1+r)^\ell\norm{z}_\infty$. This embedding implies $v\in L^{p_{\ell+1}}(\Om)$, and hence $\zeta v-z\in L^{p_{\ell+1}}(\Om)$. The weak regularity assertion in Lemma~\ref{lem:elliptic}, applied to $L_av=\zeta v-z$, gives $v\in W^{2,p_{\ell+1}}(\Om)$. Estimate \eqref{eq:w2p} then yields
		\[
		\begin{aligned}
			\norm{v}_{W^{2,p_{\ell+1}}}\le C((1+r)\norm{v}_{p_{\ell+1}}+\norm{z}_{p_{\ell+1}})\le C(1+r)^{\ell+1}\norm{z}_\infty.
		\end{aligned}
		\]
		Starting from the $W^{2,2}$ estimate above, induction gives $\norm{v}_{W^{2,2n}}\le C(1+r)^m\norm{z}_\infty$, and the Sobolev embedding $W^{2,2n}(\Om)\hookrightarrow W^{1,\infty}(\Om)$ gives
		\begin{equation}\label{eq:resolvent-gradient}
			\begin{aligned}
				\norm{\nabla R_a(r\e^{\pm\ii\theta})z}_\infty\le C(1+r)^m\norm{z}_\infty.
			\end{aligned}
		\end{equation}
		
		On the inner arc $\zeta=r_0\e^{\ii\varphi}$, $|\varphi|\le\theta$, the distance from $\zeta$ to $\sigma(L_a)$ is at least $\rho_*-r_0>0$. The spectral theorem gives $\norm{R_a(\zeta)}_{2\to2}\le(\rho_*-r_0)^{-1}$. Repeating the preceding estimates with $|\zeta|=r_0$ gives
		\[
		\begin{aligned}
			\norm{\nabla R_a(r_0\e^{\ii\varphi})z}_\infty\le C\norm{z}_\infty,\quad |\varphi|\le\theta.
		\end{aligned}
		\]
		The constants in these estimates depend only on $\Om,n,\alpha,\kappa,M$, since $\theta$ is fixed and $r_0=\rho_*/2$ depends only on $\Om$ and $\kappa$. They are independent of $a\in\mathcal A_{\kappa,M}$, $z\in L^\infty(\Om)$, and the contour parameters $r\ge r_0$ and $|\varphi|\le\theta$.
		
		For $G\in L^2(\Om;\C^n)$, define $R_a(\zeta)(-\Div G)\in L^2(\Om)$ by
		\[
		\begin{aligned}
			\int_\Om R_a(\zeta)(-\Div G)\,\overline h\,dx=\int_\Om G\cdot\nabla\overline{R_a(\overline\zeta)h}\,dx,\quad h\in L^2(\Om).
		\end{aligned}
		\]
		For $w\in\operatorname{Dom}(L_a)$, ellipticity and integration by parts give
		\[
		\begin{aligned}
			\kappa\norm{\nabla w}_2^2\le\int_\Om a|\nabla w|^2\,dx=\ip{L_aw}{w}=\norm{L_a^{1/2}w}_2^2.
		\end{aligned}
		\]
		To apply this estimate with $w=R_a(\overline\zeta)h$, we estimate $L_a^{1/2}R_a(\overline\zeta)$ by the functional calculus for the positive self-adjoint operator $L_a$. This operator corresponds to the function $s\mapsto\sqrt{s}/(\overline\zeta-s)$. The same spectral theorem used in \eqref{eq:resolvent-estimate} yields 
		\[
		\begin{aligned}
			\norm{L_a^{1/2}R_a(\overline\zeta)}_{2\to2}=\sup_{s\in\sigma(L_a)}\frac{\sqrt{s}}{|\overline\zeta-s|}\le\sup_{s\ge\rho_*}\frac{\sqrt{s}}{|\zeta-s|},
		\end{aligned}
		\]
		where we used $\sigma(L_a)\subset[\rho_*,\infty)$ and $|\overline\zeta-s|=|\zeta-s|$ for real $s$. Consequently, for every $h\in L^2(\Om)$,
		\[
		\begin{aligned}
			\norm{\nabla R_a(\overline\zeta)h}_2\le\kappa^{-1/2}\norm{L_a^{1/2}R_a(\overline\zeta)h}_2\le\kappa^{-1/2}\sup_{s\ge\rho_*}\frac{\sqrt{s}}{|\zeta-s|}\norm{h}_2.
		\end{aligned}
		\]
		
		This estimate and Cauchy--Schwarz show that the right-hand side of the defining duality identity is a bounded conjugate-linear functional of $h$. Thus, the Riesz representation theorem justifies the definition of $R_a(\zeta)(-\Div G)$ as an $L^2$ function. Taking the supremum over $\norm{h}_2=1$ gives
		\begin{equation}\label{eq:weak-resolvent}
			\begin{aligned}
				\norm{R_a(\zeta)(-\Div G)}_2\le\kappa^{-1/2}\sup_{s\ge\rho_*}\frac{\sqrt{s}}{|\zeta-s|}\norm{G}_2.
			\end{aligned}
		\end{equation}
		On the rays, \eqref{eq:resolvent-geometry} gives
		\[
		\begin{aligned}
			\sup_{s\ge\rho_*}\frac{\sqrt{s}}{|\zeta-s|}\le C_\theta\sup_{s>0}\frac{\sqrt{s}}{r+s}=\frac{C_\theta}{2\sqrt{r}}.
		\end{aligned}
		\]
		Thus, \eqref{eq:weak-resolvent} yields $\norm{R_a(\zeta)(-\Div G)}_2\le C_\theta(2\sqrt{\kappa})^{-1}r^{-1/2}\norm{G}_2$ on the rays. On the inner arc, $|\zeta-s|\ge s-r_0\ge s/2$ for $s\ge\rho_*=2r_0$, so the same estimate gives $\norm{R_a(\zeta)(-\Div G)}_2\le2(\kappa\rho_*)^{-1/2}\norm{G}_2$. If $-\Div G\in L^2(\Om)$, the identity $R_a(\zeta)^*=R_a(\overline\zeta)$ and integration by parts show that the duality definition agrees with the usual resolvent.
		
		Set $b=a_1-a_2$. Since the operators have the common domain \eqref{eq:domain}, the resolvent identity gives
		\begin{equation}\label{eq:resolvent-identity}
			\begin{aligned}
				R_{a_1}(\zeta)-R_{a_2}(\zeta)=R_{a_1}(\zeta)[-\Div(b\nabla)]R_{a_2}(\zeta).
			\end{aligned}
		\end{equation}
		For the estimate, apply \eqref{eq:weak-resolvent} with $G=b\nabla R_{a_2}(\zeta)z$. Its $L^2$ norm is at most $\norm{b}_2\norm{\nabla R_{a_2}(\zeta)z}_\infty$. Together with \eqref{eq:resolvent-gradient}, this proves
		\begin{equation}\label{eq:resolvent-difference}
			\begin{aligned}
				\norm{(R_{a_1}(\zeta)-R_{a_2}(\zeta))z}_2\le Cr^{-1/2}(1+r)^m\norm{b}_2\norm{z}_\infty
			\end{aligned}
		\end{equation}
		on the rays. On the inner arc, the bound is $C\norm{b}_2\norm{z}_\infty$. In particular, these estimates involve the $L^2$ norm of $b$ without its derivatives.
		
		Let $\Gamma$ consist of the lower ray directed outwards, the upper ray directed inwards, and the inner arc directed clockwise from angle $\theta$ to $-\theta$. Adding an outer arc from $-\theta$ to $\theta$ gives the positively oriented boundary of the truncated sector outside the disk of radius $r_0$. Set $c_0=r_0\cos\theta>0$. The function $F_T$ extends meromorphically to $\C$, with poles at $2\pi\ii j/T$, $j\in\mathbb Z\setminus\{0\}$, and a removable singularity at zero. Hence, $F_T(\lambda+\zeta)$ is holomorphic in the sector enclosed by $\Gamma$ for every $\lambda\ge0$. For $\zeta\in\Gamma$ and $T\ge T_*$,
		\begin{equation}\label{eq:shift-decay}
			\begin{aligned}
				|F_T(\lambda+\zeta)|\le\frac{(\lambda+|\zeta|)\e^{-\lambda T}\e^{-T_*\operatorname{Re}\zeta}}{1-\e^{-T_*c_0}}.
			\end{aligned}
		\end{equation}
		Indeed, $|\e^{T(\lambda+\zeta)}-1|\ge\e^{T(\lambda+\operatorname{Re}\zeta)}-1$, and $\operatorname{Re}\zeta\ge c_0$.
		
		We use the representation
		\begin{equation}\label{eq:functional-calculus}
			\begin{aligned}
				S_a^\lambda(T)=\frac1{2\pi\ii}\int_\Gamma F_T(\lambda+\zeta)R_a(\zeta)\,d\zeta.
			\end{aligned}
		\end{equation}
		The resolvent bounds and \eqref{eq:shift-decay} make the integral converge in operator norm. For each fixed $s\ge\rho_*$, close the contour by an outer arc of radius $R>2s$. On this arc, $\operatorname{Re}\zeta\ge R\cos\theta$ and $|\zeta-s|\ge R/2$, so the integral of $F_T(\lambda+\zeta)/(\zeta-s)$ over the arc tends to zero as $R\to\infty$. Cauchy's formula gives $F_T(\lambda+s)$ for the scalar integral over $\Gamma$. The integrable operator-norm bound permits interchange with the spectral integral and proves \eqref{eq:functional-calculus}. The same contour is used for every coefficient in the class.
		
		Subtracting the representations \eqref{eq:functional-calculus} and applying \eqref{eq:resolvent-difference} and \eqref{eq:shift-decay}, we bound the contribution of the two rays by a constant multiple of
		\[
		\begin{aligned}
			\e^{-\lambda T}\norm{b}_2\norm{z}_\infty\int_{r_0}^{\infty}(\lambda+r)r^{-1/2}(1+r)^m\e^{-T_*r\cos\theta}\,dr.
		\end{aligned}
		\]
		The integral is bounded by $C_*(1+\lambda)$, with $C_*$ independent of $\lambda$ and $T$. The inner arc has fixed length, and its contribution is also bounded by $C_*(1+\lambda)\e^{-\lambda T}\norm{b}_2\norm{z}_\infty$. This proves \eqref{eq:multiplier-bound}.
	\end{proof}
	
	\subsection{Pairwise comparison and choice of the input}\label{sec:global-proof}
	
	The following coefficient estimate follows from integration by parts.
	
	\begin{lemma}\label{lem:ratio}
		Let $a\in\mathcal A_{\kappa,M}$ and $q\in C^{2,\alpha}(\overline\Om)$ be real-valued, with $-\Div(a\nabla q)\ge\nu>0$ in $\Om$. For every real-valued $b\in C^1(\overline\Om)$ with zero boundary trace,
		\begin{equation}\label{eq:ratio-bound}
			\begin{aligned}
				\norm{b}_2\le\frac{2}{\kappa\nu}\norm{\Div(b\nabla q)}_2.
			\end{aligned}
		\end{equation}
		The function $q$ may have nonzero boundary values and critical points.
	\end{lemma}
	
	\begin{proof}
		Put $r=b/a$ and $D=\Div(b\nabla q)=\Div(ar\nabla q)$. Since $a\ge\kappa$ and $b$ has zero boundary trace, one has $r\in C^1(\overline\Om)\cap H_0^1(\Om)$. Integration by parts gives
		\begin{equation}\label{eq:ratio-identity}
			\begin{aligned}
				-\int_\Om rD\,dx=\int_\Om ar\nabla q\cdot\nabla r\,dx=\frac12\int_\Om a\nabla q\cdot\nabla(r^2)\,dx=\frac12\int_\Om[-\Div(a\nabla q)]r^2\,dx
				\ge\frac\nu2\norm{r}_2^2.
			\end{aligned}
		\end{equation}
		Both boundary terms vanish because $r$ has zero boundary trace. The Cauchy--Schwarz inequality implies 
		\[
		\begin{aligned}
			\frac\nu2\norm{r}_2^2\le-\int_\Om rD\,dx\le\norm{r}_2\norm{D}_2.
		\end{aligned}
		\]
		If $r=0$, then $b=0$ and the conclusion is immediate. Otherwise, division by $\norm{r}_2$ and the bound $a\le\kappa^{-1}$ yield $\norm{b}_2\le\kappa^{-1}\norm{r}_2\le\frac{2}{\kappa\nu}\norm{D}_2$, which proves \eqref{eq:ratio-bound}.
	\end{proof}
	
	\begin{proof}[Proof of Theorem~\ref{thm:global}]
		Fix $T_*>0$ and increase $B_*$ in Lemma~\ref{lem:multiplier} to at least $1$. Choose $\lambda$ so that
		\begin{equation}\label{eq:lambda-choice}
			\begin{aligned}
				\lambda\ge\max\{1,d_0^2M^2\},\quad\lambda T_*-D_0\sqrt\lambda\ge H_0,\quad H_0=\max\{\log2,\log(16B_*/\kappa)\},
			\end{aligned}
		\end{equation}
		where $d_0,D_0$ are defined in \eqref{eq:barrier-constants}. For example, take $\lambda=x^2$ with
		\begin{equation}\label{eq:explicit-rate}
			\begin{aligned}
				x\ge\max\Big\{1,d_0M,\frac{D_0+\sqrt{D_0^2+4T_*H_0}}{2T_*}\Big\}.
			\end{aligned}
		\end{equation}
		This choice depends only on the prescribed parameters.
		
		Take $a_1,a_2\in\mathcal A_{\kappa,M}$ with the same boundary values and fix $T\ge T_*$. For simplicity, let us write $u_j=u_{a_j}$, $P_j=P_{a_j}(T)$, $S_j=S_{a_j}^\lambda(T)$, $b=a_1-a_2$, and $d=P_1-P_2$. Lemma~\ref{lem:final-identities} gives $(A_{a_j}+S_j)P_j=\lambda$ for $j=1,2$. Since $(A_{a_1}-A_{a_2})P_2=-\Div(b\nabla P_2)$, subtraction yields
		\[
		\begin{aligned}
			0=(A_{a_1}+S_1)P_1-(A_{a_2}+S_2)P_2=(A_{a_1}+S_1)d-\Div(b\nabla P_2)+(S_1-S_2)P_2.
		\end{aligned}
		\]
		Rearranging yields
		\begin{equation}\label{eq:difference}
			\begin{aligned}
				\Div(b\nabla P_2)=(S_1-S_2)P_2+(A_{a_1}+S_1)d.
			\end{aligned}
		\end{equation}
		The functions $P_1,P_2,d$ belong to the common domain $H^2(\Om)\cap H_0^1(\Om)$ of the Dirichlet operators, and all terms in \eqref{eq:difference} belong to $L^2(\Om)$.
		
		The assumption that $a_1$ and $a_2$ have the same boundary values gives $b|_{\partial\Om}=0$. The choice \eqref{eq:lambda-choice} satisfies the hypotheses of Lemma~\ref{lem:positivity}, so \eqref{eq:positive-lp} gives $L_{a_2}P_2\ge\nu_\lambda=\frac{\lambda}{2}\e^{-D_0\sqrt\lambda}$. Applying Lemma~\ref{lem:ratio} with $a=a_2$, $q=P_2$, and $\nu=\nu_\lambda$, we obtain
		\begin{equation}\label{eq:global-coercivity}
			\begin{aligned}
				\norm{b}_2\le\frac{2}{\kappa\nu_\lambda}\norm{\Div(b\nabla P_2)}_2=\frac{4}{\kappa\lambda}\e^{D_0\sqrt\lambda}\norm{\Div(b\nabla P_2)}_2.
			\end{aligned}
		\end{equation}
		
		We estimate the two terms on the right-hand side of \eqref{eq:difference} separately. For the first term, \eqref{eq:multiplier-bound} with $z=P_2$ and the bound $\norm{P_2}_\infty\le1$ give
		\[
		\begin{aligned}
			\norm{(S_1-S_2)P_2}_2\le B_*(1+\lambda)\e^{-\lambda T}\norm{b}_2\norm{P_2}_\infty\le B_*(1+\lambda)\e^{-\lambda T}\norm{b}_2.
		\end{aligned}
		\]
		Substituting \eqref{eq:difference} into \eqref{eq:global-coercivity} and using this bound, we obtain
		\[
		\begin{aligned}
			\norm{b}_2
			&\le\frac{4}{\kappa\lambda}\e^{D_0\sqrt\lambda}\bigl(\norm{(S_1-S_2)P_2}_2+\norm{(A_{a_1}+S_1)d}_2\bigr)\\
			&\le\frac{4B_*}{\kappa}\frac{1+\lambda}{\lambda}\e^{D_0\sqrt\lambda-\lambda T}\norm{b}_2+\frac{4}{\kappa\lambda}\e^{D_0\sqrt\lambda}\norm{(A_{a_1}+S_1)d}_2.
		\end{aligned}
		\]
		By \eqref{eq:lambda-choice}, we have $\lambda\ge1$ and $\lambda T_*-D_0\sqrt\lambda\ge H_0$. Since $T\ge T_*$, it follows that
		\[
		\begin{aligned}
			\frac{4B_*}{\kappa}\frac{1+\lambda}{\lambda}\e^{D_0\sqrt\lambda-\lambda T}\le\frac{8B_*}{\kappa}\e^{D_0\sqrt\lambda-\lambda T_*}\le\frac{8B_*}{\kappa}\e^{-H_0}\le\frac12,
		\end{aligned}
		\]
		where the last inequality uses $H_0\ge\log(16B_*/\kappa)$. Therefore,
		\[
		\begin{aligned}
			\norm{b}_2\le\frac12\norm{b}_2+\frac{4}{\kappa\lambda}\e^{D_0\sqrt\lambda}\norm{(A_{a_1}+S_1)d}_2.
		\end{aligned}
		\]
		Subtracting $\frac12\norm{b}_2$ from both sides and multiplying by $2$ gives
		\begin{equation}\label{eq:absorbed}
			\begin{aligned}
				\norm{b}_2\le\frac{8}{\kappa\lambda}\e^{D_0\sqrt\lambda}\norm{(A_{a_1}+S_1)d}_2.
			\end{aligned}
		\end{equation}
		
		It remains to estimate $(A_{a_1}+S_1)d$ in terms of the measured final states. For every $a\in\mathcal A_{\kappa,M}$ and $v\in\operatorname{Dom}(L_a)$, the identity $L_av=-a\Delta v-\nabla a\cdot\nabla v$ gives
		\[
		\begin{aligned}
			\norm{L_av}_2\le\kappa^{-1}\norm{\Delta v}_2+M\norm{\nabla v}_2\le C_L\norm{v}_{H^2},
		\end{aligned}
		\]
		where $C_L$ depends only on $n,\kappa,M$. Also, \eqref{eq:multiplier-norm} gives $\norm{S_1}_{2\to2}\le T^{-1}\le T_*^{-1}$. Since $A_{a_1}=L_{a_1}+\lambda I$ and $d\in\operatorname{Dom}(L_{a_1})$, these estimates imply
		\[
		\begin{aligned}
			\norm{(A_{a_1}+S_1)d}_2\le\norm{L_{a_1}d}_2+\lambda\norm{d}_2+\norm{S_1d}_2\le(C_L+\lambda+T_*^{-1})\norm{d}_{H^2}.
		\end{aligned}
		\]
		The normalization $P_j=1-\e^{-\lambda T}u_j(T)$ gives $d=-\e^{-\lambda T}(u_1(T)-u_2(T))$ and $\norm{d}_{H^2}=\e^{-\lambda T}\norm{u_1(T)-u_2(T)}_{H^2}$. Substituting these two estimates into \eqref{eq:absorbed}, we obtain
		\begin{equation}\label{eq:physical-stability}
			\begin{aligned}
				\norm{a_1-a_2}_2\le\frac{8}{\kappa\lambda}\e^{D_0\sqrt\lambda-\lambda T}(C_L+\lambda+T_*^{-1})\norm{u_1(T)-u_2(T)}_{H^2}.
			\end{aligned}
		\end{equation}
		
		Since $T\ge T_*$, estimate \eqref{eq:global-stability} follows with $C_*=\frac{8}{\kappa\lambda}\e^{D_0\sqrt\lambda-\lambda T_*}(C_L+\lambda+T_*^{-1})$. The selected $\lambda$ and this constant depend only on $\Om,n,\alpha,\kappa,M,T_*$. In particular, they are independent of the coefficient pair and of $T\ge T_*$. If $u_1(T)=u_2(T)$, estimate \eqref{eq:global-stability} gives $\norm{a_1-a_2}_2=0$, and hence $a_1=a_2$ as desired.
	\end{proof}
	
	\begin{remark}\label{rem:normalization}
		The factor $\e^{-\lambda T}$ in \eqref{eq:physical-stability} comes from the relation $u_a(T)=\e^{\lambda T}W_a(T)$. For differences of the normalized profiles $W_a(T)$, this factor is absent from the stability estimate. Multiplying the initial datum and boundary input by the same constant $c>0$ multiplies the final state by $c$ and divides the corresponding stability constant by $c$.
	\end{remark}

	\section{Positive spatial profiles}\label{sec:positive-inputs}
	
	We replace the initial datum $1$ and boundary input $\e^{\lambda t}$ in Theorem~\ref{thm:global} by $k$ and $\e^{\lambda t}k$, respectively, where $k$ is a known positive spatial profile. The exponential rate can be chosen uniformly for all such profiles with a prescribed positive lower bound and a prescribed $C^{2,\alpha}$ norm bound. For $0<m\le K<\infty$, set
	\[
	\begin{aligned}
		\mathcal K_{m,K}=\{k\in C^{2,\alpha}(\overline\Om): \, k\ge m,\ \norm{k}_{C^{2,\alpha}}\le K\}.
	\end{aligned}
	\]
	We take the $C^{2,\alpha}$ norm to dominate the sup norms of the function and its derivatives through order two.
	
	\begin{corollary}\label{cor:positive-inputs}
		Fix $T_*>0$, $0<m\le K<\infty$, and the coefficient class $\mathcal A_{\kappa,M}$. There are $\lambda>0$ and $C_*<\infty$, depending only on $\Om,n,\alpha,\kappa,M,T_*,m,K$, such that the following holds. For any $k\in\mathcal K_{m,K}$ and any $a_1,a_2\in\mathcal A_{\kappa,M}$ with the same boundary values, let $u_j$ solve
		\[
		\begin{aligned}
			\begin{cases}
				\partial_tu_j-\Div(a_j\nabla u_j)=0 &\text{in }\Om\times(0,\infty),\\
				u_j=\e^{\lambda t}k &\text{on }\partial\Om\times(0,\infty),\\
				u_j(\cdot,0)=k &\text{in }\Om.
			\end{cases}
		\end{aligned}
		\]
		Then, for every $T\ge T_*$,
		\[
		\begin{aligned}
			\norm{a_1-a_2}_{L^2(\Om)}\le C_*\norm{u_1(\cdot,T)-u_2(\cdot,T)}_{H^2(\Om)}.
		\end{aligned}
		\]
		The same exponential rate and stability constant work for every $k\in\mathcal K_{m,K}$ and every admissible coefficient pair.
	\end{corollary}
	
	\begin{proof}
		Take $\lambda\ge\max\{1,d_0^2M^2\}$, where $d_0$ is defined in \eqref{eq:barrier-constants}; a further lower bound will be chosen below. Fix $k\in\mathcal K_{m,K}$. For $a\in\mathcal A_{\kappa,M}$, let $u_a$ denote the solution with initial datum $k$ and boundary input $\e^{\lambda t}k$, and set $W_a(t)=\e^{-\lambda t}u_a(t)$, $P_a(t)=k-W_a(t)$, and $f_a=-\Div(a\nabla k)+\lambda k$. The function $P_a$ satisfies $\partial_tP_a+A_aP_a=f_a$ with initial and boundary values $0$, where $A_a=L_a+\lambda I$. The solution formula and the scalar identities used in Lemma~\ref{lem:final-identities} give
		\begin{equation}\label{eq:positive-input-identities}
			\begin{aligned}
				P_a(T)=A_a^{-1}(I-\e^{-TA_a})f_a,\quad (A_a+S_a^\lambda(T))P_a(T)=f_a.
			\end{aligned}
		\end{equation}
		
		The normalized function satisfies $\partial_tW_a-\Div(a\nabla W_a)+\lambda W_a=0$, with initial and boundary values $k$. Comparison with the constant functions $0$ and $K$ gives $0\le W_a(t)\le K$. Since $0\le k\le K$, it follows that $\norm{P_a(t)}_\infty\le K$. Formula \eqref{eq:positive-input-identities} places $P_a(T)$ in $\operatorname{Dom}(L_a)$ for $T>0$. Its equation $L_aP_a(T)=f_a-\e^{-TA_a}f_a-\lambda P_a(T)$, together with $f_a\in C^\alpha(\overline\Om)$ and Lemmas~\ref{lem:elliptic} and~\ref{lem:positive-time}, gives $P_a(T)\in C^{2,\alpha}(\overline\Om)$ by the same argument as for \eqref{eq:lp-formula}. Thus, $u_a(T)$ also belongs to $C^{2,\alpha}(\overline\Om)$.
		
		Since $\Div(a\nabla k)=a\Delta k+\nabla a\cdot\nabla k$, the coefficient bounds give $\norm{\Div(a\nabla k)}_\infty\le C_0K$, where we may take $C_0=n(\kappa^{-1}+M)$. Since $\lambda\ge1$, $\norm{f_a}_\infty\le(1+C_0)K\lambda$. Let $V_a$ be the solution of $-\Div(a\nabla V_a)+\lambda V_a=0$ in $\Om$, with $V_a=1$ on $\partial\Om$, used in the proof of Lemma~\ref{lem:positivity}. It satisfies $0\le V_a\le1$ and, by \eqref{eq:barrier}, $V_a\ge\e^{-D_0\sqrt\lambda}$. The stationary function $mV_a$ solves the equation for $W_a$. On the boundary it is at most $k$, and initially $mV_a\le m\le k$. The parabolic comparison principle gives $W_a(t)\ge mV_a$.
		
		Set $q_a(T)=\e^{-TA_a}f_a$. Lemma~\ref{lem:positive-time} gives $\norm{q_a(T)}_\infty\le\e^{-\lambda T}\norm{f_a}_\infty$. Applying $A_a$ to the first identity in \eqref{eq:positive-input-identities} and substituting the definition of $f_a$, we obtain
		\begin{equation}\label{eq:positive-input-coercivity}
			\begin{aligned}
				-\Div(a\nabla(P_a(T)-k))=\lambda W_a(T)-q_a(T)\ge\lambda (m\e^{-D_0\sqrt\lambda}-(1+C_0)K\e^{-\lambda T}).
			\end{aligned}
		\end{equation}
		The estimate uses only the absolute bound for $q_a(T)$, and it holds when $f_a$ changes sign.
		
		Increase $B_*$ in Lemma~\ref{lem:multiplier} to at least $1$. Define $H_{m,K}$ and choose $\lambda$ by
		\[
		\begin{aligned}
			H_{m,K}=\max\left\{\log\frac{2(1+C_0)K}{m},\ \log\frac{16B_*K}{\kappa m}\right\},\  
			\lambda\ge\max\{1,d_0^2M^2\},\ \lambda T_*-D_0\sqrt\lambda\ge H_{m,K}.
		\end{aligned}
		\]
		Such a rate is obtained from \eqref{eq:explicit-rate} with $H_0$ replaced by $H_{m,K}$, and depends only on the stated parameters. The first term in the maximum defining $H_{m,K}$ ensures that \eqref{eq:positive-input-coercivity} gives
		\[
		\begin{aligned}
			-\Div(a\nabla(P_a(T)-k))\ge\frac{m\lambda}{2}\e^{-D_0\sqrt\lambda},\quad T\ge T_*.
		\end{aligned}
		\]
		
		Fix $T\ge T_*$ and write $P_j=P_{a_j}(T)$, $S_j=S_{a_j}^\lambda(T)$, $b=a_1-a_2$, and $d=P_1-P_2$. Since $f_{a_1}-f_{a_2}=-\Div(b\nabla k)$, subtraction of the second identity in \eqref{eq:positive-input-identities} gives
		\begin{equation}\label{eq:positive-input-difference}
			\begin{aligned}
				D:=\Div\bigl(b\nabla(P_2-k)\bigr)=(S_1-S_2)P_2+(A_{a_1}+S_1)d.
			\end{aligned}
		\end{equation}
		The function $P_2-k$ has boundary values $-k$ and belongs to $C^{2,\alpha}(\overline\Om)$. Since $b$ has zero boundary trace, Lemma~\ref{lem:ratio} applies to this profile and yields $\norm{b}_2\le\frac{4}{\kappa m\lambda}\e^{D_0\sqrt\lambda}\norm{D}_2$. Using Lemma~\ref{lem:multiplier} and $\norm{P_2}_\infty\le K$, the coefficient of $\norm{b}_2$ arising from $(S_1-S_2)P_2$ is bounded by
		\[
		\begin{aligned}
			\frac{4B_*K}{\kappa m}\frac{1+\lambda}{\lambda}\e^{D_0\sqrt\lambda-\lambda T}\le\frac{8B_*K}{\kappa m}\e^{-H_{m,K}}\le\frac12.
		\end{aligned}
		\]
		Moving this term to the left and using the bound for $A_{a_1}+S_1$ from the proof of Theorem~\ref{thm:global}, we obtain $\norm{b}_2\le(8/(\kappa m\lambda))\e^{D_0\sqrt\lambda}(C_L+\lambda+T_*^{-1})\norm{d}_{H^2}$. Since the two experiments use the same $k$, one has $d=-\e^{-\lambda T}(u_1(T)-u_2(T))$, and hence
		\[
		\begin{aligned}
			\norm{a_1-a_2}_2\le\frac{8}{\kappa m\lambda}\e^{D_0\sqrt\lambda-\lambda T}(C_L+\lambda+T_*^{-1})\norm{u_1(T)-u_2(T)}_{H^2}.
		\end{aligned}
		\]
		Replacing $\e^{-\lambda T}$ by $\e^{-\lambda T_*}$ proves the assertion, with $C_*=(8/(\kappa m\lambda))\e^{D_0\sqrt\lambda-\lambda T_*}(C_L+\lambda+T_*^{-1})$.
	\end{proof}

	\section{Time-independent inputs near a known reference}\label{sec:local}
	
	\subsection{The local pairwise estimate}
	
	Near a known reference coefficient, one can use the same spatial profile for the initial datum and the time-independent boundary input. Consider
	\begin{equation}\label{local:eq:heat}
		\begin{aligned}
			\begin{cases}
				\partial_tu-\Div(a\nabla u)=0 &\text{in }\Om\times(0,\infty),\\
				u=k &\text{on }\partial\Om\times(0,\infty),\\
				u(\cdot,0)=k &\text{in }\Om.
			\end{cases}
		\end{aligned}
	\end{equation}
	We denote its solution by $u_{a,k}$ and write $u_{a,k}(t)=u_{a,k}(\cdot,t)$.
	
	\begin{theorem}\label{local:thm:main}
		Let $a_0\in C^\infty(\overline\Om)$ be real-valued and positive. There is a real-valued $h_0\in C^\infty(\overline\Om)$ satisfying $\Div(a_0\nabla h_0)=0$ and $\inf_{\overline\Om}|\nabla h_0|>0$ with the following property. For every $T_*>0$, there are $\delta_*>0$, $\varepsilon_*>0$, and $C_*>0$ such that, whenever real positive $a_j\in C^{1,\alpha}(\overline\Om)$ satisfy
		\begin{equation}\label{local:eq:coefficient-neighborhood}
			\begin{aligned}
				\norm{a_j-a_0}_{C^{1,\alpha}}\le\delta_*,\quad a_j|_{\partial\Om}=a_0|_{\partial\Om},\quad j=1,2,
			\end{aligned}
		\end{equation}
		and real $k\in C^{2,\alpha}(\overline\Om)$ satisfies $\norm{k-h_0}_{C^{2,\alpha}}\le\varepsilon_*$, one has
		\begin{equation}\label{local:eq:main-estimate}
			\begin{aligned}
				\norm{a_1-a_2}_{L^2(\Om)}\le C_*\norm{u_{a_1,k}(T)-u_{a_2,k}(T)}_{H^2(\Om)},\quad T\ge T_*.
			\end{aligned}
		\end{equation}
		The constants depend only on $\Om,n,\alpha,a_0,h_0,T_*$ and the norm convention. The same constants work for every such coefficient pair, input, and observation time. Equality of the two final states at any one time $T\ge T_*$ implies $a_1=a_2$.
	\end{theorem}
	
	The harmonicity condition is imposed only on $h_0$; the common input $k$ may vary throughout the prescribed $C^{2,\alpha}$ neighborhood. For a constant reference, an affine profile suffices.
	
	\begin{corollary}\label{local:cor:affine}
		Let $a_0=1$ and $h_0(x)=x_1+c$, where $c\in\mathbb R$ is fixed. For every $T_*>0$, the conclusion of Theorem~\ref{local:thm:main} holds with this $h_0$.
	\end{corollary}
	
	\subsection{The final-state comparison and transport estimate}
	
	For $a\in\mathcal A_{\kappa,M}$ and real $k\in C^{2,\alpha}(\overline\Om)$, set $f_{a,k}=\Div(a\nabla k)$ and $z_{a,k}(t)=u_{a,k}(t)-k$, then $z_{a,k}$ has zero initial and boundary values and satisfies $\partial_tz_{a,k}+L_az_{a,k}=f_{a,k}$. Consequently,
	\begin{equation}\label{local:eq:mild}
		\begin{aligned}
			z_{a,k}(T)=L_a^{-1}(I-\e^{-TL_a})f_{a,k}.
		\end{aligned}
	\end{equation}
	Since $f_{a,k}\in C^\alpha(\overline\Om)$, Lemma~\ref{lem:positive-time} gives $z_{a,k}(T)\in C^{2,\alpha}(\overline\Om)$ and the uniform estimate \eqref{eq:correction-smoothing} for $T\ge T_*$. The formula applies without a boundary condition on $f_{a,k}$.
	
	Set $S_a(T)=F_T(L_a)$. The scalar identity $(s+F_T(s))(1-\e^{-Ts})/s=1$, applied to \eqref{local:eq:mild}, gives $(L_a+S_a(T))z_{a,k}(T)=f_{a,k}$. For two coefficients and the same $k$, write $z_j=z_{a_j,k}(T)$, $b=a_1-a_2$, and $d=z_1-z_2$. Since $f_{a_1,k}-f_{a_2,k}=\Div(b\nabla k)$, subtraction yields
	\begin{equation}\label{local:eq:difference-identity}
		\begin{aligned}
			\Div(b(\nabla k+\nabla z_2))=(S_{a_1}(T)-S_{a_2}(T))z_2+(L_{a_1}+S_{a_1}(T))d.
		\end{aligned}
	\end{equation}
	The case $\lambda=0$ of Lemma~\ref{lem:multiplier} gives
	\[
	\begin{aligned}
		\norm{(S_{a_1}(T)-S_{a_2}(T))z}_2\le B_*\norm{b}_2\norm{z}_\infty,\quad T\ge T_*.
	\end{aligned}
	\]
	
	\begin{lemma}\label{local:lem:transport}
		Let real $h_0\in C^{2,\alpha}(\overline\Om)$ satisfy $|\nabla h_0|\ge\mu>0$, and set
		\[
		\begin{aligned}
			\gamma=\frac{2+\norm{\Delta h_0}_\infty}{\mu^2},\quad C_0=2\e^{\gamma\osc_{\overline\Om}h_0}.
		\end{aligned}
		\]
		Choose $G_0>0$ such that $\norm{\gamma\nabla h_0\cdot\nabla r+\Delta r}_\infty\le G_0\norm{r}_{C^{2,\alpha}}$ for every $r\in C^{2,\alpha}(\overline\Om)$. If real $r\in C^{2,\alpha}(\overline\Om)$ satisfies $G_0\norm{r}_{C^{2,\alpha}}\le1$, then
		\begin{equation}\label{local:eq:transport-bound}
			\begin{aligned}
				\norm{b}_2\le C_0\norm{\Div(b(\nabla h_0+\nabla r))}_2,\quad b\in H_0^1(\Om;\mathbb R).
			\end{aligned}
		\end{equation}
	\end{lemma}
	
	\begin{proof}
		Put $X=\nabla h_0+\nabla r$ and $w=\e^{-\gamma h_0}$. Integration by parts, using the zero trace of $b$, gives
		\begin{equation}\label{local:eq:weighted-identity}
			\begin{aligned}
				\int_\Om wb [\Div(bX)]\,dx=\frac12\int_\Om w(\gamma X\cdot\nabla h_0+\Div X)b^2\,dx.
			\end{aligned}
		\end{equation}
		For $b\in H_0^1$, the identity follows by approximation with smooth compactly supported functions, since $X\in W^{1,\infty}$. The choice of $\gamma$ gives $\gamma|\nabla h_0|^2+\Delta h_0\ge2$, while the assumption on $r$ gives $|\gamma\nabla r\cdot\nabla h_0+\Delta r|\le1$. Hence $\gamma X\cdot\nabla h_0+\Div X\ge1$. Applying the Cauchy--Schwarz inequality yields
		\[
		\begin{aligned}
			\frac12(\min_{\overline\Om}w)\norm{b}_2^2\le(\max_{\overline\Om}w)\norm{b}_2\norm{\Div(bX)}_2.
		\end{aligned}
		\]
		The ratio $2\max w/\min w$ equals $C_0$, which proves \eqref{local:eq:transport-bound}.
	\end{proof}
	
	\begin{proposition}\label{local:prop:pairwise}
		Let $a_0\in C^{1,\alpha}(\overline\Om)$ be a positive function and $h_0\in C^{2,\alpha}(\overline\Om)$ be a real function satisfying
		\begin{equation}\label{local:eq:stationary-reference}
			\begin{aligned}
				\Div(a_0\nabla h_0)=0,\quad \inf_{\overline\Om}|\nabla h_0|=\mu>0.
			\end{aligned}
		\end{equation}
		For every $T_*>0$, there are $\delta_*>0$, $\varepsilon_*>0$, and $C_*>0$ such that \eqref{local:eq:main-estimate} holds for real $a_j\in C^{1,\alpha}(\overline\Om)$ and $k\in C^{2,\alpha}(\overline\Om)$ satisfying $\norm{a_j-a_0}_{C^{1,\alpha}}\le\delta_*$, $a_1|_{\partial\Om}=a_2|_{\partial\Om}$, and $\norm{k-h_0}_{C^{2,\alpha}}\le\varepsilon_*$. The constants depend only on $\Om,n,\alpha,a_0,h_0,T_*$ and the norm convention, and work for every $T\ge T_*$.
	\end{proposition}
	
	\begin{proof}
		Choose $r_*>0$ with $r_*\le\min\{1,\tfrac12\min_{\overline\Om}a_0\}$. The ball $\norm{a-a_0}_{C^{1,\alpha}}\le r_*$ then lies in $\mathcal A_{\kappa,M}$, by taking $\kappa$ sufficiently small and $M\ge\norm{a_0}_{C^{1,\alpha}}+r_*$. We fix these values of $\kappa$ and $M$ throughout the proof. The elliptic and semigroup estimates hold with the same constants for every coefficient in this ball.
		
		Write $\eta=k-h_0$. The stationary equation in \eqref{local:eq:stationary-reference} gives
		\[
		\begin{aligned}
			f_{a,k}=\Div((a-a_0)\nabla h_0)+\Div(a\nabla\eta).
		\end{aligned}
		\]
		The product rule and the $C^\alpha$ product estimate give constants $A_h,A_a>0$, depending only on the fixed coefficient class and $h_0$, such that
		\begin{equation}\label{local:eq:forcing-bound}
			\begin{aligned}
				\norm{f_{a,k}}_{C^\alpha}\le A_h\norm{a-a_0}_{C^{1,\alpha}}+A_a\norm{k-h_0}_{C^{2,\alpha}}.
			\end{aligned}
		\end{equation}
		Only the first derivatives of $a$ enter this estimate. By \eqref{eq:correction-smoothing}, $\norm{z_{a,k}(T)}_{C^{2,\alpha}}\le B_0\norm{f_{a,k}}_{C^\alpha}$ for $T\ge T_*$. Thus, for coefficients and inputs in balls of radii $\delta$ and $\varepsilon$, respectively, one has $\norm{z_{a,k}(T)}_{C^{2,\alpha}}\le q$, where $q=B_0(A_h\delta+A_a\varepsilon)$.
		
		Let $\gamma,C_0,G_0$ be as in Lemma~\ref{local:lem:transport}. The vector field in \eqref{local:eq:difference-identity} equals $\nabla h_0+\nabla(\eta+z_2)$, and $b=a_1-a_2\in H_0^1(\Om)$ because the coefficients have the same boundary values. Suppose that
		\begin{equation}\label{local:eq:simultaneous-conditions}
			\begin{aligned}
				\delta\le r_*,\quad G_0(\varepsilon+q)\le1,\quad C_0B_*q\le\tfrac12.
			\end{aligned}
		\end{equation}
		The transport estimate, \eqref{local:eq:difference-identity}, and Lemma~\ref{lem:multiplier} then give
		\[
		\begin{aligned}
			\norm{b}_2\le C_0B_*\norm{b}_2\norm{z_2}_\infty+C_0\norm{(L_{a_1}+S_{a_1}(T))d}_2\le\tfrac12\norm{b}_2+C_0\norm{(L_{a_1}+S_{a_1}(T))d}_2.
		\end{aligned}
		\]
		The uniform bound $\norm{L_ad}_2\le C_L\norm{d}_{H^2}$ and the scalar inequality $F_T(s)\le T^{-1}$ give $\norm{(L_{a_1}+S_{a_1}(T))d}_2\le(C_L+T_*^{-1})\norm{d}_{H^2}$. Since $d=u_{a_1,k}(T)-u_{a_2,k}(T)$, absorption proves \eqref{local:eq:main-estimate} with $C_*=2C_0(C_L+T_*^{-1})$.
		
		To satisfy \eqref{local:eq:simultaneous-conditions}, take all constants positive and set
		\begin{equation}\label{local:eq:radii}
			\begin{aligned}
				Q&=\min\{(2G_0)^{-1},(2C_0B_*)^{-1}\},\\
				\delta_*&=\min\{r_*,Q/(2B_0A_h)\},\\
				\varepsilon_*&=\min\{(2G_0)^{-1},Q/(2B_0A_a)\}.
			\end{aligned}
		\end{equation}
		These choices give $q\le Q$, $G_0(\varepsilon_*+q)\le1$, and $C_0B_*q\le\tfrac12$. They are independent of the coefficient pair, the chosen input, and $T\ge T_*$.
	\end{proof}
	
	The proposition assumes only equality of the two coefficient boundary values. It therefore applies when both equal the reference boundary values, as in Theorem~\ref{local:thm:main}.
	
	\subsection{Construction of the reference profile}
	
	\begin{lemma}\label{local:lem:cgo}
		Every real positive $a_0\in C^\infty(\overline\Om)$ admits a real $h_0\in C^\infty(\overline\Om)$ satisfying \eqref{local:eq:stationary-reference}, in every dimension $n\ge2$.
	\end{lemma}
	
	\begin{proof}
		Extend $\log a_0$ smoothly to a neighborhood of $\overline\Om$, multiply the extension by a smooth cutoff equal to $1$ near $\overline\Om$, and exponentiate. This gives a smooth positive extension of $a_0$ to $\mathbb R^n$, equal to $1$ outside a ball. We keep the notation $a_0$ for this extension and choose a fixed ball $B$ containing both $\overline\Om$ and the support of $q_0=-\Delta\sqrt{a_0}/\sqrt{a_0}$.
		
		The complex geometrical optics (CGO) construction in \cite[Section~5.1, proof of Lemma~3.3]{BalBonnetierMonardTriki}, valid for $n\ge2$, gives solutions $v_\rho=\e^{\rho\cdot x}(1+\psi_\rho)$ of $\Delta v_\rho+q_0v_\rho=0$ for sufficiently large $|\rho|$, with $\rho\cdot\rho=0$. Since $q_0\in C_c^\infty(\mathbb R^n)$, \cite[Propositions~3.1 and~3.3, equation~(25)]{BalUhlmann} gives $\norm{\psi_\rho}_{C^1(\overline B)}\le C/|\rho|$, where $B$ contains the entire support of the extended potential, and the constants depend only on the fixed extension and $B$. The function $U_\rho=a_0^{-1/2}v_\rho$ satisfies $\Div(a_0\nabla U_\rho)=0$.
		
		Differentiation gives
		\begin{equation}\label{local:eq:cgo-gradient}
			\begin{aligned}
				\sqrt{a_0}\nabla U_\rho=\e^{\rho\cdot x}(\rho+\phi_\rho),\quad \phi_\rho=\rho\psi_\rho+\nabla\psi_\rho-(1+\psi_\rho)\nabla\log\sqrt{a_0}.
			\end{aligned}
		\end{equation}
		Hence, $\norm{\phi_\rho}_{L^\infty(\Om)}\le C_1$ independently of sufficiently large $|\rho|$. Choose $\rho=\tau(e_1+\ii e_2)$, where $e_1,e_2$ are the first two coordinate vectors, and put $h_0=\operatorname{Re}U_\rho$. Since $a_0$ is real, $h_0$ is $a_0$-harmonic. Moreover,
		\[
		\begin{aligned}
			\sqrt{a_0}\nabla h_0=\e^{\tau x_1}[\tau(e_1\cos(\tau x_2)-e_2\sin(\tau x_2))+\operatorname{Re}(\e^{\ii\tau x_2}\phi_\rho)].
		\end{aligned}
		\]
		The leading real vector has length $\tau$. Taking $\tau$ above the CGO threshold and with $\tau\ge2C_1$ gives
		\[
		\begin{aligned}
			|\nabla h_0(x)|\ge\frac{\tau\e^{\tau x_1}}{2\sqrt{a_0(x)}}\ge\frac{\tau\e^{\tau\min_{\overline\Om}x_1}}{2\sqrt{\max_{\overline\Om}a_0}}>0.
		\end{aligned}
		\]
		The solution is smooth on a neighborhood of $\overline\Om$ by interior elliptic regularity; see \cite[Theorem~9.19]{GilbargTrudinger2001}. This proves the lemma.
	\end{proof}
	
	\begin{proof}[Proof of Theorem~\ref{local:thm:main} and Corollary~\ref{local:cor:affine}]
		Choose $h_0$ by Lemma~\ref{local:lem:cgo} and apply Proposition~\ref{local:prop:pairwise} for each $T_*>0$. The profile is fixed using only the known reference. When $a_0=1$, take $h_0=x_1+c$. It satisfies \eqref{local:eq:stationary-reference} since $\Delta h_0=0$ and $|\nabla h_0|=1$, so the proposition also proves the corollary.
	\end{proof}
	
	Adding a constant to $h_0$ preserves \eqref{local:eq:stationary-reference}. After adding a sufficiently large constant and, if necessary, reducing the input radius, the center profile and every admissible input are positive.
	
	\section{Optimality of the final-data norm}\label{sec:optimality}
	
	Throughout this section, let $L=-\Delta$ with $\operatorname{Dom}(L)=H^2(\Om)\cap H_0^1(\Om)$, and write $\langle v,w\rangle=\int_\Om v\overline w\,dx$ for the $L^2$ inner product.
	
	Poincar\'e's inequality and the elliptic estimate in Lemma~\ref{lem:elliptic}, applied to $L=-\Delta$, give
	\[
	\begin{aligned}
		\norm{L^{-1}f}_{H^2}\le C\norm{f}_2,\quad f\in L^2(\Om),
	\end{aligned}
	\]
	where $C$ depends only on $\Om$.
	
	\subsection{The exponential experiment}\label{sec:exponential-sharpness}
	
	Fix $\lambda,T>0$ and set $\mathcal F_{\lambda,T}(a)=u_a(T)$, where $u_a$ solves the exponential experiment with initial datum $1$ and boundary input $\e^{\lambda t}$.
	
	\begin{theorem}\label{thm:exponential-sharpness}
		Let $n\ge2$ and fix $\lambda,T>0$. For every $0\le s<2$, there are no $\delta>0$ and $C<\infty$ such that
		\[
		\begin{aligned}
			\norm{a_1-a_2}_2\le C\norm{\mathcal F_{\lambda,T}(a_1)-\mathcal F_{\lambda,T}(a_2)}_{H^s}
		\end{aligned}
		\]
		for all real positive $a_j\in C^{1,\alpha}(\overline\Om)$ satisfying $a_j|_{\partial\Om}=1$ and $\norm{a_j-1}_{C^{1,\alpha}}\le\delta$.
	\end{theorem}
	
	\begin{proof}
		Let $U=u_1$ be the solution for the constant coefficient $1$. Extend $U$ to negative times by setting $U(x,t)=1$ for $t<0$. Since $U(t)\to1$ in $L^2(\Om)$ as $t\downarrow0$ and $\Delta1=0$, the extension satisfies $\partial_tU-\Delta U=0$ in distributions across $t=0$ in the interior of $\Om$. Interior regularity for the heat equation therefore gives bounded spatial derivatives of every fixed order of $U$ and $\partial_tU$ on compact subsets of $\Om$ for $0\le t\le T$. This also gives $\nabla U(0)=0$.
		
		We first prove convergence of the directional difference quotient in $H^2$. Fix real $b\in C_c^\infty(\Om)$, set $a_\varepsilon=1+\varepsilon b$, and define $Q_\varepsilon(t)=(u_{a_\varepsilon}(t)-U(t))/\varepsilon$ for small nonzero $\varepsilon$. Then $Q_\varepsilon$ has zero initial and boundary values and satisfies $\partial_tQ_\varepsilon+L_{a_\varepsilon}Q_\varepsilon=f_b(t)$, where $f_b(t)=\Div(b\nabla U(t))$. Let $Q_0$ solve this problem with $L$ in place of $L_{a_\varepsilon}$. The source is smooth, compactly supported in a fixed subset of $\Omega$, and uniformly bounded for $0\le t\le T$ in $\norm{f_b(t)}_2+\norm{L^jf_b(t)}_2$ for every fixed $j\ge1$. The mild formula gives
		\[
		\begin{aligned}
			L^jQ_0(t)=\int_0^t\e^{-(t-r)L}L^jf_b(r)\,dr.
		\end{aligned}
		\]
		Semigroup contraction and successive interior elliptic estimates give uniform bounds on every fixed spatial derivative of $Q_0$ on compact subsets; see \cite[Theorem~9.19]{GilbargTrudinger2001}.
		
		Put $R_\varepsilon=Q_\varepsilon-Q_0$ and $g_b(t)=\Div(b\nabla Q_0(t))$. Subtraction gives $\partial_tR_\varepsilon+L_{a_\varepsilon}R_\varepsilon=\varepsilon g_b(t)$, with zero initial and boundary values. The source $g_b(t)$ is supported in $\supp b$ and is uniformly bounded in $H^2(\Om)$ for $0\le t\le T$. In particular, its boundary trace is zero, so $g_b(t)\in H^2(\Om)\cap H_0^1(\Om)=\operatorname{Dom}(L_{a_\varepsilon})$. For fixed $b$ and small $\varepsilon$, Lemma~\ref{lem:elliptic} uniformly compares the $H^2$ norm with $\norm{v}_2+\norm{L_{a_\varepsilon}v}_2$. The semigroup contracts both terms of the latter norm, so
		\[
		\begin{aligned}
			R_\varepsilon(T)=\varepsilon\int_0^T\e^{-(T-t)L_{a_\varepsilon}}g_b(t)\,dt,\quad \norm{R_\varepsilon(T)}_{H^2}\le C_{b,T}|\varepsilon|.
		\end{aligned}
		\]
		We write $D\mathcal F_{\lambda,T}(1)b=Q_0(T)$ for this directional derivative. The convergence is in $H^2$ for each fixed $b$, and $C_{b,T}$ may depend on $b$.
		
		By \eqref{eq:mild}, the function $P(T)=1-\e^{-\lambda T}U(T)$ is nonzero and has zero boundary trace: the scalar multiplier $(1-\e^{-Ts})/s$ is strictly positive on the spectrum of $L+\lambda I$. Hence, $U(T)$ is not constant. Connectedness of $\Om$ implies that $\nabla U(T)$ is nonzero at some interior point. The flow-box theorem for the smooth vector field $\nabla U(T)$ gives local coordinates $x=\Phi(r,y)$ in a neighborhood $V\Subset\Om$, with $\partial_r\Phi(r,y)=\nabla U(T)(\Phi(r,y))$ and $y\in\mathbb R^{n-1}$. Since $n\ge2$, define $\phi(\Phi(r,y))=y_1$. Differentiation in $r$ and $y_1$ gives $\nabla\phi\cdot\nabla U(T)=0$ and $\nabla\phi(\Phi(r,y))\cdot\partial_{y_1}\Phi(r,y)=1$, so $|\nabla\phi|>0$ on $V$. Put $Y=\nabla\phi/|\nabla\phi|^2$, so $Y\cdot\nabla\phi=1$, and choose nonzero real $\beta\in C_c^\infty(V)$. Define $b_N=\beta\cos(N\phi)$ on $V$ and extend it by zero to $\Om$. Integration by parts along $Y$ gives $\norm{b_N}_2^2\to\tfrac12\norm{\beta}_2^2>0$ as $N\to \infty$.
		
		Set $q_N=D\mathcal F_{\lambda,T}(1)b_N$. Its source is $f_N(t)=\Div(b_N\nabla U(t))$. At $t=T$, the orthogonality condition on $\phi$ gives $f_N(T)=g\cos(N\phi)$, where $g=\Div(\beta\nabla U(T))\in C_c^\infty(V)$ is independent of $N$. Also $f_N(0)=0$. Integrating the Duhamel formula by parts in time yields
		\begin{equation}\label{eq:terminal-linear-ibp}
			\begin{aligned}
				q_N=L^{-1}f_N(T)-\int_0^T L^{-1}\e^{-(T-t)L}\Div(b_N\nabla\partial_tU(t))\,dt.
			\end{aligned}
		\end{equation}
		For each fixed $N$, the source is continuously differentiable in $L^2$ and compactly supported, which justifies this identity in $L^2$.
		
		Let $A\in H^2(V)$ be supported in $\supp\beta$. Since $Y\cdot\nabla\phi=1$, we have $Y\cdot\nabla\sin(N\phi)=N\cos(N\phi)$ and $Y\cdot\nabla\cos(N\phi)=-N\sin(N\phi)$. Two integrations by parts give
		\[
		\begin{aligned}
			\int_V A\cos(N\phi)\,dx=-\frac1N\int_V\Div(AY)\sin(N\phi)\,dx=-\frac1{N^2}\int_V\Div\bigl(Y\Div(AY)\bigr)\cos(N\phi)\,dx.
		\end{aligned}
		\]
		There are no boundary terms because $A$ is compactly supported in $V$. The expression $\Div(Y\Div(AY))$ contains derivatives of $A$ of order at most two, with coefficients bounded on $\supp\beta$. The Cauchy--Schwarz inequality gives
		\[
		\Big|\int_V A\cos(N\phi)\,dx\Big|\le CN^{-2}\norm{A}_{H^2(V)},
		\]
		where $C$ depends only on the fixed support and $\phi$. Smooth approximation justifies the identities for $A\in H^2(V)$. For $\psi\in L^2(\Om)$, set $v=L^{-1}\psi$. Applying this observation to $g\overline v$ and using the Dirichlet elliptic estimate gives
		\[
		\begin{aligned}
			|\langle L^{-1}f_N(T),\psi\rangle|=|\langle g\cos(N\phi),v\rangle|\le CN^{-2}\norm{v}_{H^2}\le CN^{-2}\norm{\psi}_2.
		\end{aligned}
		\]
		The amplitude is supported in $V$, so there are no boundary terms; smooth approximation justifies the integrations for $v\in H^2$. Also, $\norm{L^{-1}f_N(T)}_{H^2}\le C\norm{g}_2$ independently of $N$.
		
		Denote the integral in \eqref{eq:terminal-linear-ibp} by $r_N$. For $\tau>0$, let $v_\tau=L^{-1}\e^{-\tau L}\psi$. Since $V\Subset\Om$ and $Lv_\tau=\e^{-\tau L}\psi$, interior elliptic regularity gives
		\[
		\begin{aligned}
			\norm{v_\tau}_{H^3(V)}\le C(\norm{\e^{-\tau L}\psi}_{H^1(\Om)}+\norm{v_\tau}_2)\le C(1+\tau^{-1/2})\norm{\psi}_2;
		\end{aligned}
		\]
		see \cite[Theorem~9.19]{GilbargTrudinger2001}. The second inequality follows from the Dirichlet inverse estimate, $L^2$ contraction, and
		\[
		\begin{aligned}
			\norm{\nabla\e^{-\tau L}\psi}_2=\norm{L^{1/2}\e^{-\tau L}\psi}_2\le\bigl(\sup_{\xi>0}\sqrt{\xi}\e^{-\tau\xi}\bigr)\norm{\psi}_2\le C\tau^{-1/2}\norm{\psi}_2.
		\end{aligned}
		\]
		Pair $r_N$ with $\psi$ and integrate the divergence by parts. Applying the preceding oscillatory estimate with $A=\beta\nabla\partial_tU(t)\cdot\nabla\overline{v_{T-t}}$ gives
		\[
		\begin{aligned}
			|\langle r_N,\psi\rangle|\le CN^{-2}\int_0^T\norm{v_{T-t}}_{H^3(V)}\,dt\le C_TN^{-2}\norm{\psi}_2.
		\end{aligned}
		\]
		The spatial derivatives of $\partial_tU$ are uniformly bounded on $\supp\beta\Subset V$, and $(T-t)^{-1/2}$ is integrable.
		
		For the $H^2$ bound, let $G\in L^2(\Om;\C^n)$ and $\tau>0$. Define $\e^{-\tau L}\Div G$ by duality:
		\[
		\begin{aligned}
			\langle\e^{-\tau L}\Div G,\psi\rangle=-\int_\Om G\cdot\nabla\overline{\e^{-\tau L}\psi}\,dx,\quad \psi\in L^2(\Om).
		\end{aligned}
		\]
		The preceding gradient estimate gives
		\[
		\begin{aligned}
			|\langle\e^{-\tau L}\Div G,\psi\rangle|\le C\tau^{-1/2}\norm{G}_2\norm{\psi}_2.
		\end{aligned}
		\]
		Thus the pairing defines an $L^2$ function and agrees with the usual semigroup action when $\Div G\in L^2(\Om)$. The Dirichlet inverse estimate gives
		\[
		\begin{aligned}
			\norm{L^{-1}\e^{-\tau L}\Div G}_{H^2}\le C\norm{\e^{-\tau L}\Div G}_2\le C\tau^{-1/2}\norm{G}_2.
		\end{aligned}
		\]
		Apply this estimate with $G=b_N\nabla\partial_tU(t)$ and $\tau=T-t$. The $L^2$ norm of $G$ is bounded independently of $N$ and $t$, and $(T-t)^{-1/2}$ is integrable. Hence the integral defining $r_N$ converges in $H^2(\Om)\cap H_0^1(\Om)$ and satisfies $\norm{r_N}_{H^2}\le C_T$. Together with the preceding $L^2$ bounds, this gives
		\[
		\begin{aligned}
			\norm{q_N}_2\le C_TN^{-2},\quad \norm{q_N}_{H^2}\le C_T.
		\end{aligned}
		\]
		Sobolev interpolation yields $\norm{q_N}_{H^s}\le C_{T,s}N^{s-2}$ for $0\le s<2$.
		
		Suppose the asserted nonlinear Lipschitz estimate held. For each fixed $N$, take $a_1=1+\varepsilon b_N$ and $a_2=1$. They satisfy the coefficient assumptions for sufficiently small $|\varepsilon|$, with the allowable size depending on $N$. Divide by $|\varepsilon|$ and let $\varepsilon\to0$. The $H^2$ directional convergence gives $\norm{b_N}_2\le C\norm{q_N}_{H^s}$. Letting $N\to\infty$ contradicts $\norm{b_N}_2^2\to\tfrac12\norm{\beta}_2^2>0$. The directional limit is taken with $N$ fixed throughout.
	\end{proof}
	
	\subsection{The affine experiment}\label{sec:sharpness}
	
	\begin{theorem}\label{affine:thm:sharpness}
		Let $n\ge2$, fix $T>0$, and set $h_0(x)=x_1+c$ with fixed $c\in\mathbb R$. For every $0\le s<2$, there are no $\delta>0$ and $C<\infty$ such that
		\begin{equation}\label{affine:eq:excluded-estimate}
			\begin{aligned}
				\norm{a_1-a_2}_2\le C\norm{u_{a_1,h_0}(T)-u_{a_2,h_0}(T)}_{H^s}
			\end{aligned}
		\end{equation}
		for all real positive $a_j\in C^{1,\alpha}(\overline\Om)$ satisfying $a_j|_{\partial\Om}=1$ and $\norm{a_j-1}_{C^{1,\alpha}}\le\delta$.
	\end{theorem}
	
	For this fixed profile, define $\mathcal F_T(a)=u_{a,h_0}(T)$. The reference solution at $a=1$ is stationary: $u_{1,h_0}(t)=h_0$. Hence the source of the linearized equation in a direction $b\in C_c^\infty(\Om;\mathbb R)$ is $\Div(b\nabla h_0)=\partial_1b$, and
	\begin{equation}\label{affine:eq:linearization}
		\begin{aligned}
			D\mathcal F_T(1)b=L^{-1}(I-\e^{-TL})\partial_1b.
		\end{aligned}
	\end{equation}
	The $H^2$ directional convergence follows from the argument in the proof of Theorem~\ref{thm:exponential-sharpness}, with $U=h_0$. In this case, the source $\partial_1b$ is time-independent, smooth, and compactly supported, so it satisfies all the regularity conditions used in that argument.
	
	\begin{proof}[Proof of Theorem~\ref{affine:thm:sharpness}]
		Choose a nonzero real function $\beta\in C_c^\infty(\Om)$ and set $b_N=\beta\sin(Nx_2)$ for positive integers $N$. Since $\sin^2(Nx_2)=(1-\cos(2Nx_2))/2$,
		\[
		\begin{aligned}
			\norm{b_N}_2^2=\frac12\norm{\beta}_2^2-\frac12\int_\Om\beta^2\cos(2Nx_2)\,dx.
		\end{aligned}
		\]
		Integration by parts in $x_2$, with no boundary term because $\beta$ is compactly supported, gives
		\[
		\begin{aligned}
			\Big|\int_\Om\beta^2\cos(2Nx_2)\,dx\Big|=\Big|-\frac1{2N}\int_\Om\partial_2(\beta^2)\sin(2Nx_2)\,dx\Big|\le\frac1{2N}\norm{\partial_2(\beta^2)}_{L^1}\longrightarrow0.
		\end{aligned}
		\]
		Thus, $\norm{b_N}_2^2\to\tfrac12\norm{\beta}_2^2>0$. In particular, there is an integer $N_0$ such that $\norm{b_N}_2\ge c_\beta:=\frac12\norm{\beta}_2>0$ for all $N\ge N_0$.
		
		Put $f_N=\partial_1b_N=(\partial_1\beta)\sin(Nx_2)$ and $w_N=L^{-1}f_N$. For $\varphi\in L^2(\Om)$, let $v=L^{-1}\varphi$. Self-adjointness of $L^{-1}$ and two integrations by parts in $x_2$ give
		\[
		\begin{aligned}
			\langle w_N,\varphi\rangle=\langle f_N,v\rangle=\int_\Om(\partial_1\beta)\overline v\sin(Nx_2)\,dx=-\frac1{N^2}\int_\Om\partial_2^2((\partial_1\beta)\overline v)\sin(Nx_2)\,dx.
		\end{aligned}
		\]
		The product $(\partial_1\beta)\overline v$ belongs to $H^2(\Om)$ and is compactly supported in $\Om$, so these integrations by parts produce no boundary terms. The product rule and the elliptic estimate for $L^{-1}$ imply
		\[
		\begin{aligned}
			|\langle w_N,\varphi\rangle|\le CN^{-2}(\norm{v}_2+\norm{\partial_2v}_2+\norm{\partial_2^2v}_2)\le CN^{-2}\norm{v}_{H^2}
			\le CN^{-2}\norm{\varphi}_2.
		\end{aligned}
		\]
		Taking the supremum over $\norm{\varphi}_2=1$ yields $\norm{w_N}_2\le CN^{-2}$.
		
		By \eqref{affine:eq:linearization}, write
		\[
		\begin{aligned}
			q_N:=D\mathcal F_T(1)b_N=(I-\e^{-TL})w_N.
		\end{aligned}
		\]
		The semigroup is an $L^2$ contraction, preserves $\operatorname{Dom}(L)$, and commutes with $L$ on this domain. Consequently,
		\[
		\begin{aligned}
			\norm{q_N}_2\le2\norm{w_N}_2\le CN^{-2},\quad \norm{Lq_N}_2=\norm{(I-\e^{-TL})f_N}_2\le2\norm{f_N}_2\le2\norm{\partial_1\beta}_2.
		\end{aligned}
		\]
		The elliptic estimate therefore gives $\norm{q_N}_{H^2}\le C(\norm{Lq_N}_2+\norm{q_N}_2)\le C$. Interpolation between $L^2$ and $H^2$ now yields
		\begin{equation}\label{affine:eq:oscillatory-bound}
			\begin{aligned}
				\norm{D\mathcal F_T(1)b_N}_{H^s}=\norm{q_N}_{H^s}\le C_s\norm{q_N}_2^{1-s/2}\norm{q_N}_{H^2}^{s/2}\le C_{T,s}N^{s-2},\quad 0\le s<2,
			\end{aligned}
		\end{equation}
		where $C_{T,s}$ is independent of $N$.
		
		Suppose that \eqref{affine:eq:excluded-estimate} holds with a constant $C_{\mathrm{Lip}}$ independent of the coefficient pair. Fix $N$ and take $a_1=1+\varepsilon b_N$ and $a_2=1$. Both coefficients equal $1$ near $\partial\Om$, and
		\[
		\begin{aligned}
			\norm{a_1-1}_{C^{1,\alpha}}=|\varepsilon|\norm{b_N}_{C^{1,\alpha}}\longrightarrow0
		\end{aligned}
		\]
		as $\varepsilon\to0$. Hence both coefficients are admissible for sufficiently small nonzero real $\varepsilon$. Applying the assumed estimate and dividing by $|\varepsilon|$ gives
		\[
		\begin{aligned}
			\norm{b_N}_2\le C_{\mathrm{Lip}}\Big\|\frac{\mathcal F_T(1+\varepsilon b_N)-\mathcal F_T(1)}{\varepsilon}\Big\|_{H^s}.
		\end{aligned}
		\]
		Keeping $N$ fixed, the directional differentiability established above gives
		\[
		\begin{aligned}
			\frac{\mathcal F_T(1+\varepsilon b_N)-\mathcal F_T(1)}{\varepsilon}
			\longrightarrow D\mathcal F_T(1)b_N
			\quad\text{in }H^2(\Om)
		\end{aligned}
		\]
		as $\varepsilon\to0$, and hence also in $H^s(\Om)$. Passing to this limit, using \eqref{affine:eq:oscillatory-bound} and the lower bound for $\norm{b_N}_2$, we obtain
		\[
		\begin{aligned}
			0<c_\beta
			\le\norm{b_N}_2
			\le C_{\mathrm{Lip}}\norm{D\mathcal F_T(1)b_N}_{H^s}
			\le C_{\mathrm{Lip}}C_{T,s}N^{s-2},\quad N\ge N_0.
		\end{aligned}
		\]
		We now let $N\to\infty$. Since $s<2$, the right-hand side tends to zero, contradicting $c_\beta>0$.
	\end{proof}

	\section*{Statements and declarations}
	
	\para{Data availability statement} No datasets were generated or analyzed in this study.
	
	\para{Conflict of interest} The authors declare no conflict of interest.
	
	\para{Acknowledgments} Y.-H. Lin is partially supported by the National Science and Technology Council (NSTC), Taiwan, under the project 113-2115-M-A49-017-MY3. He also acknowledges financial support from the Alexander von Humboldt Foundation through the Henriette Herz Scouting Programme, hosted by Universit\"at Duisburg-Essen, Germany. The main results, arguments, and statements of this paper were developed without the use of AI. AI tools were used for auxiliary mathematical discussions and for assistance with parts of the writing. The authors take full responsibility for the final mathematical statements and proofs.

	\bibliographystyle{alpha}
	\bibliography{refs}
	
\end{document}